\documentclass[a4paper,12pt]{article} 
\usepackage{IEEEtrantools}

\usepackage{mathtools, nccmath}
\usepackage{latexsym}
\usepackage{amsfonts}
\usepackage{amssymb}
\usepackage{eqnarray, amsmath,amsfonts,amsthm}
\usepackage{mathrsfs}
\usepackage{enumerate}
\usepackage{dsfont}
\usepackage{hyperref}
\usepackage{graphicx}
\usepackage{epstopdf}
\usepackage{tikz}
\usepackage{physics}

\usetikzlibrary{shapes,arrows.meta,positioning}
\usepackage{cleveref}
\usepackage{tikz-cd}
\usepackage{psfrag}

\usepackage{multirow}
\usepackage{tikz}
\usetikzlibrary{arrows,decorations.markings}
\usepackage{bookmark}
\usepackage{hyperref}
\hypersetup{
colorlinks  = true,
citecolor = blue
}

\newtheorem{thm}{Theorem}[section]
\newtheorem{lmma}[thm]{Lemma}
\newtheorem{ppsn}[thm]{Proposition}
\newtheorem{crlre}[thm]{Corollary}

\theoremstyle{definition}
\newtheorem{dfn}[thm]{Definition}

\newtheorem{rmrk}[thm]{Remark}

\newcommand{\bbn}{\mathbb{N}}

\def \qed { \mbox{}\hfill
$\Box$\vspace{1ex}}

\providecommand{\Dref}[1]{Definition~\ref{#1}}
\providecommand{\Tref}[1]{Theorem~\ref{#1}}
\providecommand{\Lref}[1]{Lemma~\ref{#1}}
\providecommand{\Pref}[1]{Proposition~\ref{#1}}
\providecommand{\Sref}[1]{Section~\ref{#1}}
\providecommand{\Eref}[1]{(\ref{#1})} 

\title{Holomorphic Jet Modules and Holomorphic Connections for Noncommutative Complex Curves}

\begin{document}

\tikzset{->-/.style={decoration={
markings,
mark=at position #1 with {\arrow{>}}},postaction={decorate}}}
  \tikzset{-<-/.style={decoration={
  markings,
  mark=at position #1 with {\arrow{<}}},postaction={decorate}}}

\author{\sc{Indranil Biswas, Satyajit Guin, Pradip Kumar}}
\date{}
\maketitle

\hypersetup{linkcolor=blue}
\bigskip

\begin{abstract}
We extend Atiyah's holomorphic jet bundle formalism to holomorphic vector bundles over noncommutative algebras endowed with a bigraded differential calculus truncated at bidegree $(1,1)$; such structures are referred to as noncommutative complex curves. For a holomorphic vector bundle $(E,\,\overline{\nabla}_E)$ over such an algebra $\mathcal{A}$, we construct a canonical holomorphic structure $\overline{\nabla}_J$ on the first jet module $J_E^1\,$, making the jet sequence
\[
0\longrightarrow \Omega^{1,0}(\mathcal{A})\otimes_{\mathcal A}E
\longrightarrow J_E^1
\longrightarrow E
\longrightarrow 0
\]
exact in the holomorphic category. The assignment $(E,\,\overline\nabla_E)\,\rightsquigarrow\,(J_E^1,\,\overline\nabla_J)$ defines an endofunctor on the category of holomorphic vector bundles over $\mathcal{A}$. We define the notion of holomorphic connection in this setting and prove that a holomorphic vector bundle admits a holomorphic connection if and only if the above jet sequence splits in the holomorphic category, or equivalently, if and only if its Atiyah class vanishes. This yields a noncommutative analogue of Atiyah's classical correspondence for Riemann surfaces. Finally, we specialize to the quantum projective line $\mathbb{CP}_q^1\,$ and determine when $\overline{\nabla}_J$ defines a bimodule connection, assuming that $\overline{\nabla}_E$ does so.
\end{abstract}
\bigskip

{\bf AMS Subject Classification No.:} {58B34, 46L87, 14A22}



\section{Introduction}\label{sec1}

Let $E$ be a holomorphic vector bundle over a complex manifold $X$; denote its holomorphic
cotangent bundle by $\Omega_X^1$. The first jet sequence
\[
0\longrightarrow \Omega_X^1\otimes E\longrightarrow J^1(E)\longrightarrow E\longrightarrow 0
\]
is a fundamental object in complex geometry, encoding the first-order holomorphic data of $E$ in a canonical and functorial manner. Its splitting as a sequence of holomorphic vector bundles is intimately tied to the existence of a holomorphic connection on $E$. Atiyah showed in his foundational work \cite{Ati57} that a holomorphic vector bundle $E$ admits a holomorphic connection if and only if the jet sequence splits holomorphically, or equivalently, if and only if the obstruction class $\mathrm{At}(E)\in H^1(X,\,\Omega_X^1\otimes\mathrm{End}(E))$, called the Atiyah class for $E$, vanishes. This cohomological obstruction governs the deviation of $E$ from admitting a holomorphic connection and carries a deep geometric content. The Atiyah class represents the Chern classes of $E$ in Hodge cohomology, making it a bridge between the holomorphic and topological data of the bundle. The situation is particularly transparent on compact complex curves. Here, $\Omega_X^1$ is the canonical line bundle of degree $2g-2$, where $g$ is the genus of $X$. The group $H^1(X,\,\Omega_X^1\otimes\mathrm{End}(E))$ is finite-dimensional and computable via Serre duality, and the jet sequence provides an efficient framework for studying extensions, holomorphic differential operators, and the moduli space of bundles with connection.

In noncommutative complex geometry, which is in its infancy compared to the classical theory, a systematic holomorphic theory of jet bundles has not been developed, even for basic noncommutative spaces. Recently, the notion of jet modules over noncommutative spaces has been introduced in \cite{MSa23} and \cite{FMW25}. The aim of this article is to initiate the study of such a theory for holomorphic vector bundles over noncommutative spaces. Here we restrict our attention to ``noncommutative complex curves". We begin by explaining this non-standard terminology, which is motivated by analogy with the classical theory of Riemann surfaces from the perspective of complex structure, without implying any deeper structural correspondence.

Recall that a differential structure on a noncommutative algebra $\mathcal{A}$ is given by a pair $(\Omega^\bullet,\,d)$, where $\Omega^\bullet\,=\,\bigoplus_{n\ge 0}\,\Omega^n(\mathcal{A})$ is a differential graded algebra and $d\,:\,\Omega^n(\mathcal{A})\,\longrightarrow\,\Omega^{n+1}(\mathcal{A})$ is a differential that satisfies the graded Leibniz rule. A complex structure on a noncommutative algebra $\mathcal{A}$ (see \Sref{sec2}) equipped with a differential calculus is given by a decomposition $\Omega^n(\mathcal{A})\,=\,\bigoplus_{p+q=n}\Omega^{p,q}(\mathcal{A})$ for all $n\,\in\,\bbn$ together with differentials $\partial\,:\,\Omega^{\bullet,\bullet}(\mathcal{A})\,\longrightarrow\,\Omega^{\bullet+1,\bullet}(\mathcal{A})$ and $\overline\partial\,:\,\Omega^{\bullet,\bullet}(\mathcal{A})\,\longrightarrow\,\Omega^{\bullet,\bullet+1}(\mathcal{A})$ such that $d\,=\,\partial+\overline\partial$. A noncommutative complex curve is a noncommutative algebra $\mathcal{A}$ over $\mathbb{C}$ equipped with a complex structure truncated at the bidegree $(1,\,1)$. This means that
\[
\Omega^\bullet\ =\ \mathcal{A}\oplus\big(\Omega^{1,0}(\mathcal{A})\oplus\Omega^{0,1}(\mathcal{A})\big)\oplus\Omega^{1,1}(\mathcal{A})
\]
and all higher-degree forms vanish. We also need to impose additional assumptions that $\Omega^{1,0}(\mathcal{A})$ is finitely generated and projective as a left module over $\mathcal{A}$, and the wedge map
\[
\wedge\ :\ \Omega^{0,1}(\mathcal{A})\otimes_{\mathcal{A}}\Omega^{1,0}(\mathcal{A})\
\longrightarrow\ \Omega^{1,1}(\mathcal{A})
\]
is an $\mathcal{A}$-bimodule isomorphism (see \Dref{NC complex curve}). The latter condition is slightly weaker than the notion of a \emph{factorisable calculus} considered in \cite{BM20}. The motivating example is the quantum projective line $\mathbb{CP}_q^1$ equipped with the left covariant $2D$-calculus obtained by restricting Woronowicz's $3D$-calculus on $SU_q(2)$ \cite{KLvS11}.

In noncommutative geometry, a smooth vector bundle over a noncommutative algebra $\mathcal{A}$ is defined to be a finitely generated projective left $\mathcal{A}$-module. If $\mathcal{A}$ is equipped with a complex structure, we say that $E$ is a holomorphic vector bundle over $\mathcal{A}$ if it is finitely generated projective left $\mathcal{A}$-module and there is a flat $\overline\partial$-connection $\overline{\nabla}_E\,:\,E\,\longrightarrow\,\Omega^{0,1}(\mathcal{A})\otimes_{\mathcal{A}}E$. Thus, a holomorphic vector bundle over $\mathcal{A}$ is a pair $(E,\,\overline{\nabla}_E)$ of the above type. Unlike classical complex geometry, in the noncommutative setting the theory of holomorphic vector bundles is still in its initial stage of its development. For example, in the case of the  quantum projective line $\mathbb{CP}_q^1$,\, it is unknown whether every holomorphic vector bundle over it decomposes as a direct sum of holomorphic line bundles, like in the classical case which was shown by Grothendieck. To emphasize the level of difficulty, we remark that each line bundle over the quantum projective line admits multiple holomorphic structures up to gauge equivalence --- unlike the classical case --- as constructed recently in \cite{GLRM25} and further analyzed in \cite{BGK26}.
\bigskip

\noindent\textbf{Main Results.} Given a holomorphic vector bundle $(E,\,\overline\nabla_E)$ over a noncommutative complex curve $\mathcal{A}$, we define its first jet module $J_E^1$ from the holomorphic point of view following \cite{MSa23}. Using the assumption that $\Omega^{1,0}(\mathcal{A})$ is finitely generated projective as left module over $\mathcal{A}$, it is deduced that $J_E^1$ again becomes a finitely generated projective left $\mathcal{A}$-module, and hence $J_E^1$ is a smooth vector bundle over $\mathcal{A}\,$. Then, we construct a holomorphic structure on $J_E^1$ lifting the $\overline\partial$-connection $\overline\nabla_E$ on $E$ to a $\overline\partial$-connection $\overline\nabla_J$ on $J_E^1\,$. The association $(E,\,\overline{\nabla}_E)\,\rightsquigarrow\,(J_E^1,\,\overline{\nabla}_J)$ is functorial. Furthermore, the associated jet sequence
\begin{IEEEeqnarray}{lCl}\label{intro split}
0\,\longrightarrow\, \Omega^{1,0}(\mathcal{A})\otimes_{\mathcal{A}}E\,\longrightarrow\, J^1_E\,
\longrightarrow\, E\,\longrightarrow \,0
\end{IEEEeqnarray}
is a short exact sequence in the holomorphic category. Its extension class $$At(E)\ \in\ \mathrm{Ext}_{\mathrm{hol}}^1(E,\,\Omega^{1,0}\otimes_{\mathcal{A}}E)$$ is the Atiyah class
for $(E,\,\overline{\nabla}_E)$.

A natural question that arises is when $At(E)$ vanishes, or equivalently, when the short exact sequence in \eqref{intro split} splits in the holomorphic category. In the classical case, the class $At(E)$ vanishes precisely when the holomorphic vector bundle $E$ admits a holomorphic connection. We define the notion of a holomorphic connection (see \Dref{hol con}) in the setting of holomorphic vector bundles over noncommutative complex curves, and our main result is that a holomorphic vector bundle over a noncommutative complex curve admits a holomorphic connection if and only if the holomorphic jet sequence in \Eref{intro split} splits in the holomorphic category. This provides a quantum analogue of the classical correspondence between holomorphic connections and holomorphic splittings of the jet sequence for Riemann surfaces \cite{Ati57}.

Finally, we specialize to the quantum projective line $\mathbb{CP}_q^1$ which exhibits a special feature. Suppose that $E$ is a $\mathcal{A}(\mathbb{CP}_q^1)$-bimodule, where $\mathcal{A}(\mathbb{CP}_q^1)$ is the coordinate algebra of the quantum projective line, and $E$ is equipped with a holomorphic structure $\overline{\nabla}_E$ that is a bimodule connection (for example, line bundles over the quantum projective line \cite{KLvS11}). We identify the compatibility condition between certain bimodule maps that ensures that the lifted holomorphic structure $\overline{\nabla}_J$ on the jet module $J_E^1$ is again a bimodule connection. Note that being a bimodule connection is a property of a left connection, rather than additional independent data. An application of holomorphic connection is also shown in this case, namely, we show that, under the $SU_q(2)$-invariance, the holomorphic jet sequence on the line bundle $\mathcal{L}_n$ over the quantum projective line splits holomorphically if and only if $n\,=\,0$, which is in direct parallelism with Atiyah's result for the case of the Riemann sphere.
\medskip

The organization of the article is as follows. \Cref{sec2} reviews complex structures and holomorphic vector bundles over noncommutative algebras. In \Cref{Sec3}, we construct the holomorphic jet bundle and establish the functoriality. \Cref{Sec4} introduces the notion of holomorphic connection in this setting and proves the equivalence between holomorphic connections and holomorphic splittings of the jet sequence. \Cref{sec5} treats the bimodule setting over the quantum projective line $\mathbb{CP}_q^1$ and presents an application to the line bundles over it.
\bigskip

\noindent\textbf{Acknowledgements.} S. Guin thanks Réamonn Ó Buachalla for a useful discussion. I. Biswas is partially supported by a J. C. Bose Fellowship (JBR/2023/000003).


\section{Preliminaries}\label{sec2}

Let $\mathcal{A}$ be a (noncommutative) unital algebra over $\mathbb{C}$. A differential calculus over $\mathcal{A}$ is a pair $(\Omega^\bullet,\,d)$ such that $\Omega^\bullet\,=\,\bigoplus_{n=0}^\infty\Omega^n$, with $\Omega^0\,:=\,\mathcal{A}$, is a graded associative algebra, and $d\,:\,\Omega^n\,\longrightarrow\,\Omega^{n+1}$ is a linear map satisfying $d^2\,=\,0$ and
\[
d(\omega\eta)\ =\ (d\omega)\eta+(-1)^n\omega\,d\eta
\]
for all $\omega\,\in\,\Omega^n$ and $\eta\,\in\,\Omega^\bullet$. We have $\Omega^\bullet$ is generated as an algebra by $\mathcal{A}$ and $d\mathcal{A}$, equivalently,
\[
\Omega^n\ =\ \mathrm{span}_{\mathbb{C}}\{a_0da_1\wedge\ldots\wedge da_n:a_j\in\mathcal{A}\}
\]
for all $n\,\in\,\bbn$. We use `$\wedge$' to denote the multiplication between elements of a differential calculus when both are of order greater than $0$. A differential map between two differential calculi $(\Omega^\bullet,\,d)$ and $(\widetilde{\Omega}^\bullet,\,\widetilde{d})$, defined over the same algebra $\mathcal{A}$, is a bimodule map $\Phi\,:\,\Omega^\bullet\,\longrightarrow\,\widetilde{\Omega}^\bullet$ that intertwines the differentials.

If $\mathcal{A}$ is a $\ast$-algebra, we call a differential calculus $(\Omega^\bullet,\,d)$ to be a $\ast$-differential calculus if there exists an involution $\ast\,:\,\Omega^\bullet\,\longrightarrow\,\Omega^\bullet$ such that
\begin{enumerate}[$(i)$]
\item $(d\omega)^*\,=\,d(\omega^*)$ for all $\omega\,\in\,\Omega^\bullet$, and

\item $(\omega\wedge\eta)^*\,=\,(-1)^{k\ell}\,\eta^*\wedge\omega^*$ for all $\omega\,\in\,\Omega^k$ and $\eta\,\in\,\Omega^\ell$.
\end{enumerate}

We now recall the definition of a noncommutative complex structure.

\begin{dfn}[\cite{BM20}]
By an almost complex structure on $(\Omega^\bullet,\,d)$ we mean a map $J\,:\,\Omega^n\,\longrightarrow\,\Omega^n$ for all $n$ such that the following four conditions hold:
\begin{enumerate}[$(i)$]
\item $J(\xi\wedge\eta)\,=\,J(\xi)\wedge\eta+\xi\wedge J(\eta)$ for all $\xi,\,\eta\,\in\,\Omega^\bullet$;
\item $J$ is identically $0$ on $\mathcal{A}$;
\item $J^2\,=\, -\mathrm{id}$ on $\Omega^1$; and
\item $J(\xi^*)\,=\,(J\xi)^*$ for $\xi\,\in\,\Omega^1$.
\end{enumerate}
\end{dfn}

When there is an almost complex structure on $(\Omega^\bullet,\,d)$, we have
\[
\Omega^1\ =\ \Omega^{1,0}\oplus\Omega^{0,1}
\]
where
\[
\Omega^{1,0}\,:=\,\{\xi\in\Omega^1\,:\,J\xi=i\xi\}\,,\,\,
\,\Omega^{0,1}\,:=\,\{\xi\in\Omega^1\,:\,J\xi=-i\xi\}\,,
\]
that is, $\Omega^1$ splits as a direct sum of the $\pm \sqrt{-1}$ eigenspaces of $J$. This is extended to more general forms by
\[
\Omega^{p,q}\ :=\ \{\xi\,\in\,\Omega^{p+q}\,\big\vert\,\,J\xi\,=\,(p-q)\,i\,\xi\}
\]
for integers $p,\,q\,\ge\, 0$. As $J$ is a derivation, we have $\Omega^{p,q}\wedge\Omega^{p^\prime,q^\prime}\,\subseteq\,\Omega^{p+p^\prime,q+q^\prime}$, and a direct sum decomposition $\Omega^n\,=\,\bigoplus_{p+q=n}\,\Omega^{p,q}$. Associated with this direct sum we have projections $\pi^{p,q}\,:\,\Omega^{p+q}\,\longrightarrow\,\Omega^{p,q}$, and using these, we have operators analogous
to the classical $\partial$ and $\overline\partial$ defined by
\[
\partial\ =\ \pi^{p+1,q}\circ d\ :\ \Omega^{p,q}\ \longrightarrow\ \Omega^{p+1,q}\,,
\]
\[
\overline\partial\ =\ \pi^{p,q+1}\circ d\ :\ \Omega^{p,q}\ \longrightarrow\ \Omega^{p,q+1}.
\]

\begin{lmma}[\cite{BM20}]
The following five conditions are equivalent:
\begin{enumerate}[$(i)$]
\item $0\,=\,\overline\partial^2\,:\,\mathcal{A}\, \longrightarrow\,\Omega^2$;
\item $0\,=\,\partial^2\,:\,\mathcal{A} \, \longrightarrow\,\Omega^2$;
\item $d\,=\,\partial+\overline\partial\,:\,\Omega^1\, \longrightarrow\,\Omega^2$;
\item $d\Omega^{1,0}\ \subseteq\ \Omega^{2,0}\oplus\Omega^{1,1}$;
\item $d\Omega^{0,1}\ \subseteq\ \Omega^{1,1}\oplus\Omega^{0,2}$.
\end{enumerate}
An almost complex structure $J$ is called integrable if the above equivalent conditions hold.
\end{lmma}

Suppose $(\Omega^\bullet,\,d)$ is a differential $\ast$-calculus on a $\ast$-algebra $\mathcal{A}$ with an integrable almost complex structure. Then $d\,=\,\partial+\overline\partial$ on all $\Omega^{p,q}$, or equivalently, $$d(\Omega^{p,q})\ \subseteq\ \Omega^{p+1,q}\oplus\Omega^{p,q+1}$$ (see \cite[Cor. 7.4]{BM20}).

\begin{ppsn}[\cite{BM20}]
Suppose $(\Omega^\bullet,\,d)$ is a differential $*$-calculus with an integrable almost complex structure $J$. Then $\partial$ and $\overline\partial$ are graded derivations with $\partial^2\,=\,0$,\, $\overline\partial^2\,=\,0$ and $\partial\overline\partial+\overline\partial\partial\,=\,0$. Also, $(\overline\partial\xi)^*\,=\,\partial(\xi^*)$ and $(\partial\xi)^*\,=\,\overline\partial(\xi^*)$ for all $\xi\,\in\,\Omega^\bullet$.
\end{ppsn}

\begin{dfn}
Let $\mathcal{A}$ be a noncommutative unital $\ast$-algebra over $\mathbb{C}$ equipped with a differential $\ast$-calculus $(\Omega^{\bullet},\,d)$. A \emph{complex structure} on $\mathcal{A}$ is an integrable almost complex structure.
\end{dfn}

We write $(\Omega^{\bullet,\bullet},\,\partial,\,\overline{\partial})$ for the associated complex structure on $\mathcal{A}$. The algebra of holomorphic elements in $\mathcal{A}$ is defined as
\[
\mathcal{O}(\mathcal{A})\ :=\ \mathrm{ker}\{\overline\partial\ :\ \mathcal{A}\, \longrightarrow\,\Omega^{0,1}\}.
\]
Note that using the Leibniz rule, this is indeed an algebra over $\mathbb{C}$.

\begin{dfn}[\cite{BM20}]
A differential calculus with a complex structure is said to satisfy the factorisability condition if
\[
\wedge\ :\ \Omega^{0,q}\otimes_{\mathcal{A}}\Omega^{p,0}\ \longrightarrow\ \Omega^{p,q}\,,\quad\wedge\ :\ \Omega^{p,0}\otimes_{\mathcal{A}}\Omega^{0,q}\ \longrightarrow\ \Omega^{p,q}
\]
are bimodule isomorphisms for all $p,\,q\,\ge\, 0$.
\end{dfn}

We note that this is true for all classical complex manifolds.
\medskip

Now, we recall holomorphic vector bundles in noncommutative geometry. First recall that by a \emph{smooth vector bundle} over a noncommutative algebra we mean a finitely generated projective left module over the algebra.

\begin{dfn}[{\cite{KLvS11}, \cite{BM20}}]\label{holomorphic vb}
Let $\mathcal{A}$ be a noncommutative algebra equipped with a complex structure $(\Omega^{\bullet,\bullet},\,\partial,\,\overline{\partial})$. A \emph{holomorphic vector bundle} over $\mathcal A$ consists of a finitely generated projective left $\mathcal A$-module $E$ together with a left $\overline{\partial}$-connection
\[
\overline\nabla_E\ :\ E\ \longrightarrow\ \Omega^{0,1}\otimes_{\mathcal A}E
\]
that satisfies the Leibniz rule, which states as follows:
\[
\overline\nabla_E(ae)\ =\ a\,\overline\nabla_E(e)+\overline{\partial} a\otimes e
\]
for all $a\,\in\, \mathcal A$ and $e\,\in\, E$, and the corresponding curvature vanishes, meaning
\[
\overline\nabla_E^{\,2}\ =\ 0.
\]
The pair $(E,\,\overline{\nabla}_E)$ is called a holomorphic vector bundle over $\mathcal A$.
\end{dfn}

\begin{dfn}[{\cite{BM20}}]\label{holomorphic vb morphism}
Let $(E,\,\overline\nabla_E)$ and $(\widetilde E,\,\overline\nabla_{\widetilde E})$ be holomorphic vector bundles over $\mathcal A$. A \emph{morphism} from $(E,\,\overline\nabla_E)$ to $(\widetilde E,\,\overline\nabla_{\widetilde E})$ is a left $\mathcal A$-linear map
\[
\phi\ :\ E\ \longrightarrow\ \widetilde E
\]
satisfying the condition
\[
(\mathrm{id}\otimes \phi)\circ \overline\nabla_E\ =\
\overline\nabla_{\widetilde E}\circ \phi.
\]
\end{dfn}

If $\Omega^{0,1}$ and $\Omega^{0,2}$ are flat right $\mathcal{A}$-modules, the category of holomorphic vector bundles over $\mathcal{A}$ is an abelian category \cite[Proposition 7.15]{BM20}.

\begin{dfn}\label{holomorphic subbundle}
Let $(E,\,\overline\nabla_E)$ be a holomorphic vector bundle over $\mathcal A$. A finitely generated projective left $\mathcal{A}$-submodule $F\,\subseteq\, E$ is called a holomorphic subbundle if
\[
\overline{\nabla}_E(F)\ \subseteq\ \Omega^{0,1}\otimes_{\mathcal{A}}F\,.
\]
\end{dfn}

Given a holomorphic vector bundle $(E,\,\overline\nabla_E)$ over $\mathcal{A}$, since $\overline\nabla_E$ is a flat connection, there is the following complex of vector spaces
\[
0\,\longrightarrow\, E\,\overset{\overline{\nabla}_E}{\longrightarrow}\, \Omega^{0,1}\otimes_{\mathcal{A}}E\,\overset{\overline{\nabla}_E}{\longrightarrow}\,\Omega^{0,2}\otimes_{\mathcal{A}}E\,\overset{\overline{\nabla}_E}{\longrightarrow}\,\ldots
\]
\begin{dfn}\label{hol sections}
The $j$-th cohomology group of the above complex is denoted by $H^j(E,\,\overline{\nabla}_E)$. In particular, the zero-th cohomology group $H^0(E,\,\overline{\nabla}_E)$, which is a left $\mathcal{O}(\mathcal{A})$-module, is called the space of holomorphic sections of $E$.
\end{dfn}
\bigskip

\noindent\textbf{Our framework:} In this article, we consider noncommutative algebras that are equipped with complex structures in the above sense such that
\begin{IEEEeqnarray}{lCl}\label{our dfn}
\Omega^{p,q}\ =\ 0\ \,\mbox{ when }\  p\,>\,1\ \mbox{ or }\ q\,>\,1.
\end{IEEEeqnarray}
This means that
\[
\Omega^{\bullet}\ =\ \Omega^0\,\oplus\,\big(\Omega^{1,\,0}\oplus\Omega^{0,1}\big)\,\oplus\,\Omega^{1,\,1}.
\]

\begin{dfn}\label{NC complex curve}
A noncommutative unital algebra $\mathcal{A}$ equipped with a complex structure $(\Omega^{\bullet,\bullet},\, \partial,\, \overline{\partial})$ is called a noncommutative complex curve if \Eref{our dfn} is satisfied, and moreover $\Omega^{1,0}$ is finitely generated projective as a left $\mathcal{A}$-module with
\[
\wedge\ :\ \Omega^{0,1}\otimes_{\mathcal{A}}\Omega^{1,0}\ \longrightarrow\ \Omega^{1,1}
\]
an $\mathcal{A}$-bimodule isomorphism.
\end{dfn}

In the classical case, Riemann surfaces are equipped with such differential calculus. For a genuine noncommutative example, consider the quantum projective line $\mathbb{CP}_q^1$ equipped with the left covariant $2D$-calculus obtained by restricting Woronowicz's $3D$-calculus on $SU_q(2)$. Both $\Omega^{0,1}$ and $\Omega^{1,0}$ are rank-one finitely generated projective bimodules over the coordinate algebra $\mathcal{A}(\mathbb{CP}_q^1)$, whereas, $\Omega^{1,1}$ is a free rank-one bimodule over $\mathcal{A}(\mathbb{CP}_q^1)$. We refer to \cite{KLvS11} (see also \cite{KM} and \cite{BM20}) for details. It is known that the wedge map is indeed a bimodule isomorphism \cite[Example 7.23]{BM20}.
\smallskip

We remark that the definition of a noncommutative complex curve is motivated by the classical case from the complex structure point of view, without implying any deeper structural correspondence.


\section{Holomorphic jet bundles over noncommutative complex curves}\label{Sec3}

For finitely generated projective modules over noncommutative algebras, the notion of the first jet module has been recently introduced in \cite{MSa23} (see also \cite{FMW25}). However, for our purposes, we need a very minor modification of it.

\subsection{The quantum jet modules and jet bimodules}\label{Sec3.1}

Let $\mathcal{A}$ be a noncommutative algebra equipped with a complex structure $(\Omega^{\bullet,\bullet},\,\partial,\,\overline{\partial})$. Following \cite{MSa23}, we define the first jet module of $\mathcal A$ by
\[
J_{\mathcal A}^{1,0}\ :=\ \mathcal A\oplus_{\mathbb C}\Omega^{1,0},
\qquad
j^{1,0}(x)\ :=\ x+\partial x, \quad x\,\in \,\mathcal A,
\]
where the left $\mathcal A$-action is given by
\[
a\triangleright x\ :=\ (a+\partial a)x,
\qquad a\triangleright \omega\ :=\ a\omega
\]
for all $a,\,x\,\in\, \mathcal A$ and $\omega\,\in\, \Omega^{1,0}$. Here, the direct sum is only $\mathbb C$-linear; more precisely, it is not a direct sum of left $\mathcal A$-modules.

Let $E$ be a finitely generated projective left $\mathcal A$-module, so $E$ may be thought of as a smooth vector bundle over $\mathcal{A}$. Throughout the paper, ``f.g.p.'' stands for ``finitely generated projective''. We define the following:
\begin{IEEEeqnarray}{lCl}\label{first}
J_E^{1,0}\ :=\ J_{\mathcal A}^{1,0}\otimes_{\mathcal A}E
\ =\ E\oplus_{\mathbb C}(\Omega^{1,0}\otimes_{\mathcal A}E).
\end{IEEEeqnarray}
The left $\mathcal A$-action is given by
\[
a\triangleright e\ :=\ ae+\partial a\otimes e,
\qquad
a\triangleright(\omega\otimes f)\ :=\ a\omega\otimes f,
\]
for all $a\,\in\, \mathcal A$,
$\omega\,\in\, \Omega^{1,0}$ and $e,f\,\in\, E$. Again, the decomposition in \eqref{first} is only as of a complex vector spaces.
\medskip

\noindent\textbf{Notation:} In the rest of the paper, we write $J_E^1$ in place of $J_E^{1,0}$.
\medskip

\noindent There is a natural short exact sequence of left $\mathcal A$-modules
\begin{equation}\label{ejs}
0\ \longrightarrow\ \Omega^{1,0}\otimes_{\mathcal A}E\ \longrightarrow\ J_E^1\
\longrightarrow\ E\ \longrightarrow\ 0,
\end{equation}
which we call the \emph{jet sequence} (see Remark 2.12 in \cite{FMW25} for its uniqueness).

\begin{lmma}\label{general}
Let $E$ be a f.g.p. left $\mathcal A$-module. If $\Omega^{1,0}$ is finitely generated as a left $\mathcal A$-module, then $J_E^1$ is finitely generated. Moreover, $J_E^1$ is projective as a left $\mathcal A$-module if and only if $\Omega^{1,0}\otimes_{\mathcal A}E$ is projective. In particular,
$J_E^1$ is projective as a left $\mathcal A$-module
whenever the $\mathcal{A}$-bimodule $\Omega^{1,0}$ is projective as a left $\mathcal A$-module.
\end{lmma}

\begin{prf}
This follows from the combination of Lemma 2.21 and Lemma 2.23 of \cite{FMW25}.\qed
\end{prf}

Consequently, if $\Omega^{1,0}$ is f.g.p. as a left $\mathcal A$-module, then $J_E^1$ is again a f.g.p. left $\mathcal A$-module, in which case it is a smooth vector bundle over $\mathcal A$. In this case, the jet sequence in \eqref{ejs} splits via a left $\mathcal A$-linear map $s\,:\,E\,\longrightarrow\, J_E^1\,$. Writing
\[
s(e)\ =\ (e,\,\nabla e),
\]
one obtains a map $\nabla\,:\,E\,\longrightarrow\, \Omega^{1,0}\otimes_{\mathcal A}E$. The left $\mathcal A$-action on $J_E^1$ shows that $\nabla$ is a left $\partial$-connection on $E$.
\medskip

Now suppose that $E$ is additionally a $\mathcal A$-bimodule. To equip $J_E^1$ with a compatible right $\mathcal A$-action, we need to impose the following assumption (see \cite{MSa23}).
\medskip

\noindent\textbf{Hypothesis 1:} The $\mathcal A$-bimodule $E$ is equipped with an $\mathcal A$-bimodule map (generalized braiding)
\[
\sigma_E\ :\ E\otimes_{\mathcal A}\Omega^{1,0}\ \longrightarrow\ \Omega^{1,0}\otimes_{\mathcal A}E.
\]
\smallskip

Classically, one takes $\sigma_E$ to be the flip map which is a bimodule isomorphism. In the noncommutative setting, such a map is an extra structure. The right action on $J_E^1$ is defined by
\[
e\triangleleft a\ :=\ ea+\sigma_E(e\otimes \partial a)\,,
\qquad
(\omega\otimes e)\triangleleft a\ :=\ \omega\otimes ea,
\]
for all $e\,\in\, E$ and $\omega\,\in\, \Omega^{1,0}$. A natural question that arises here is whether the jet sequence splits as $\mathcal A$-bimodules. The condition required for this is the existence of a bimodule $\partial$-connection.

\begin{dfn}[\cite{BM20}]\label{bimod con}
Let $E$ be an $\mathcal A$-bimodule together with a bimodule map
\[
\sigma_E\ :\ E\otimes_{\mathcal A}\Omega^{1,0}\ \longrightarrow
\ \Omega^{1,0}\otimes_{\mathcal A}E\,.
\] 
A left $\sigma_E$-bimodule $\partial$-connection on $E$ is a left $\partial$-connection
\[
\nabla_E\ :\ E\ \longrightarrow\  \Omega^{1,0}\otimes_{\mathcal A}E
\]
such that
\[
\nabla_E(ea)\ =\ \nabla_E(e)a+\sigma_E(e\otimes \partial a)
\]
for all $e\,\in\, E$ and $a\,\in\, \mathcal A$.
\end{dfn}

Note that such a map $\sigma_E$ is uniquely determined by the above identity. Therefore, being a bimodule connection is a property of a left connection, and not an additional independent data (see \cite[Lemma 3.67]{BM20}). If $\sigma_E$ is invertible, then a left bimodule connection can be converted into a right bimodule connection; see \cite[Lemma 3.70]{BM20}. Note that since the bimodule map $\sigma_E$ plays the role of a ``flip'' in the commutative case, it will generally be invertible in typical noncommutative geometric settings, although this need not always be the case.
\medskip

\noindent\textbf{Hypothesis 2:}\, The $\mathcal A$-bimodule $E$ admits a $\sigma_E$-bimodule
$\partial$-connection.
\medskip

Proposition 5.1 in \cite{MSa23} gives the following.
\begin{lmma}\label{bimod splitting}
Bimodule splittings of the jet sequence are in one-to-one correspondence with bimodule $\partial$-connections $\nabla_E$, via
\[
j_E^1\ :\ E\ \longrightarrow\ J_E^1,
\qquad
j_E^1(e)\ :=\ e+\nabla_E(e).
\]
\end{lmma}

Both the above hypotheses are satisfied, for example, in the case of line bundles over the quantum projective line $\mathbb{CP}_q^1\,$.

\begin{dfn}\label{de1}
Let $(E,\,\nabla_E,\,\sigma_E)$ and $(E,\,\widetilde{\nabla}_E,\,\widetilde{\sigma}_E)$ be two choices of bimodule connections on the same $\mathcal A$-bimodule $E$. We say that the corresponding quantum jet bimodules $J_E^1$ and $\widetilde{J}_E^1$ are \emph{isomorphic} if there exists an $\mathcal A$-bimodule isomorphism
\[
\Phi\ :\ J_E^1\ \longrightarrow\ \widetilde{J}_E^1
\]
satisfying the condition
\[
\Phi\circ j_E^1\ =\ \widetilde{j}_E^1.
\]
\end{dfn}

Definition \ref{de1} and Lemma 5.2 in \cite{MSa23} gives the following.
\begin{lmma}
For a given $E$, the quantum jet bimodule $J^1_E$ is independent of the particular choice of bimodule connection. More precisely, if $J_E^1$ and $\widetilde{J}_E^1$ are constructed respectively from $(\nabla_E,\,\sigma_E)$ and $(\widetilde{\nabla}_E,\widetilde{\sigma}_E)$ on $E$, then $J_E^1\cong\widetilde{J}_E^1$ as jet bimodules.
\end{lmma}

The following terminology shall be used throughout the article.
\smallskip

\noindent\emph{\textbf{Terminology\,:}}
When we write \emph{jet bundle} $J_E^1\,$, this means we consider only the left module structure on it. When we are considering the bimodule structure on $J_E^1\,$, we explicitly mention \emph{jet bimodule}.


\subsection{Lift of holomorphic structure to the jet bundle}\label{Sec3.2}

In this section, we construct a holomorphic structure on the first jet bundle $J_E^1$ associated to any holomorphic vector bundle $E$ over $\mathcal{A}$, where $\mathcal{A}$ is a noncommutative complex curve (\Dref{NC complex curve}). We show that the construction is functorial.

By definition, the complex structure on a noncommutative complex curve $\mathcal{A}$ is assumed to satisfy two conditions, namely $\Omega^{1,0}$ is f.g.p. as left $\mathcal{A}$-module, so that $J_E^1$ becomes a smooth vector bundle over $\mathcal{A}$ by \Lref{general}, and we have the wedge product ``$\wedge$'' induces a $\mathcal{A}$-bimodule isomorphism
\begin{eqnarray}\label{mu}
\mu\ :\ \Omega^{0,1}\otimes_{\mathcal A}\Omega^{1,0}\ \xrightarrow{\,\ \cong\,\ }\ \Omega^{1,1}\,,
\qquad
\mu(\omega\otimes \eta)\ =\ \omega\wedge \eta.
\end{eqnarray}
Given a holomorphic vector bundle $(E,\,\overline\nabla_E)$ over $\mathcal A$, recall from \Cref{Sec3.1} the associated first jet module
\[
J_E^1\ :=\ E\oplus_{\mathbb C}(\Omega^{1,0}\otimes_{\mathcal A}E)
\]
and the left $\mathcal A$-action on $J_E^1\,$:
\begin{equation}\label{eq:left-action-jet}
a\cdot (e,\,\eta)\ :=\ (ae,\,\partial a\otimes e+a\eta),
\qquad a\,\in\, \mathcal A,\, \,(e,\,\eta)\,\in\, J_E^1.
\end{equation}
We call $J^1_E$ the jet bundle. Moreover, there is associated the short exact sequence of left $\mathcal{A}$-modules
\[
0\ \longrightarrow\ \Omega^{1,0}\otimes_{\mathcal A}E \longrightarrow\ J_E^1
\ \longrightarrow\ E\ \longrightarrow\ 0
\]
called the jet sequence.

Construct the map
\begin{equation}\label{eq:DE-def}
D_E\ :=\ (\overline{\partial}\otimes \mathrm{id}_E) -(\wedge\otimes \mathrm{id}_E)
\circ (\mathrm{id}_{\Omega^{1,0}}\otimes \overline\nabla_E)\
:\ \Omega^{1,0}\otimes_{\mathcal A}E\ \longrightarrow\ \Omega^{1,1}\otimes_{\mathcal A}E.
\end{equation}
Now the map $\mu$ in \eqref{mu} produces the following map:
\begin{equation}\label{new def}
\widetilde{D}_E\ :=\ (\mu^{-1}\otimes \mathrm{id}_E)\circ D_E
\end{equation}
whose target space is $\Omega^{0,1}\otimes_{\mathcal A}\Omega^{1,0}\otimes_{\mathcal A}E$.
Next, we define
\[
\beta_E  \ :=\ (\partial\otimes \mathrm{id}_E)\circ \overline\nabla_E
\ :\ E\ \longrightarrow\ \Omega^{1,1}\otimes_{\mathcal A}E,
\]
and set
\begin{equation}\label{eq:beta-def}
\widetilde\beta_E \ :=\ (\mu^{-1}\otimes \mathrm{id}_E)\circ \beta_E\ :\
E\ \longrightarrow\ \Omega^{0,1}\otimes_{\mathcal A}\Omega^{1,0}\otimes_{\mathcal A}E.
\end{equation}
We have the diagram of maps
\[
\begin{tikzcd}[row sep=large, column sep=small] 
J_E^1  & \equiv & E \arrow[d, "\overline{\nabla}_E"] \arrow[drr, "\widetilde\beta_E"] & \oplus_{\mathbb{C}} & \Omega^{1,0}\otimes_{\mathcal{A}} E \arrow[d, "\widetilde D_E"]\\ 
\Omega^{0,1}\otimes_{\mathcal{A}} J_E^1 & \equiv & \Omega^{0,1}\otimes_{\mathcal{A}} E & \oplus_{\mathbb{C}} & \Omega^{0,1}\otimes_{\mathcal{A}}\Omega^{1,0}\otimes_{\mathcal{A}} E 
\end{tikzcd}
\]
The main result of this section is the following.

\begin{ppsn}\label{holomorphic structure}
The map
\[
\overline\nabla_J\,:\,J_E^1\,\longrightarrow\, \Omega^{0,1}\otimes_{\mathcal A}J_E^1\,,
\qquad
\overline\nabla_J(e,\,\eta)\,:=\,\bigl(\overline\nabla_E(e),\,\widetilde\beta_E(e)+\widetilde D_E(\eta)\bigr),
\]
defines a left $\overline\partial$-connection on $J_E^1\,$:
\begin{equation}\label{eq:nablaJ-def}
\overline\nabla_J\bigl(a\cdot (e,\,\eta)\bigr)\ =\
a\,\overline\nabla_J(e,\,\eta)+\overline\partial a\otimes (e,\,\eta)
\end{equation}
for all $a\,\in\, \mathcal A$ and $(e,\,\eta)\,\in\, J_E^1$. Consequently, if $(E\,,\,\overline\nabla_E)$ is a holomorphic vector bundle over a noncommutative complex curve, then so is $(J_E^1,\,\overline\nabla_J)$.
\end{ppsn}

\begin{proof}
Fix $a\,\in\, \mathcal A$ and $(e,\,\eta)\,\in\, J_E^1$. Recall from \eqref{eq:left-action-jet} that
\[
a\cdot (e,\,\eta)\ =\ (ae,\,\partial a\otimes e+a\eta).
\]
Hence, we have
\[
\overline\nabla_J\bigl(a\cdot (e,\,\eta)\bigr)\ =\
\Bigl(\overline\nabla_E(ae),\,\widetilde\beta_E(ae)+\widetilde D_E(\partial a\otimes e+a\eta)\Bigr).
\]
Since $\overline\nabla_E$ is a left $\overline\partial$-connection,
\[
\overline\nabla_E(ae)\ =\ a\,\overline\nabla_E(e)+\overline{\partial} a\otimes e.
\]
Therefore, to prove \eqref{eq:nablaJ-def} it suffices to show that
\begin{equation}\label{required}
\widetilde\beta_E(ae)+\widetilde D_E(\partial a\otimes e+a\eta)\
=\ \overline{\partial} a\otimes \eta+a\,\widetilde\beta_E(e)+a\,\widetilde D_E(\eta).
\end{equation}
Applying $\mu\otimes \mathrm{id}_E$, this is equivalent to
\begin{equation}\label{required1}
\beta_E(ae)+D_E(\partial a\otimes e+a\eta)\ =\
(\mu\otimes \mathrm{id}_E)(\overline{\partial} a\otimes \eta)+a\,\beta_E(e)+a\,D_E(\eta).
\end{equation}

We will now verify \eqref{required1}. First, using the Leibniz rule for $\overline\nabla_E$, we obtain
the following:
\begin{align}
\beta_E(ae)\
&=\ (\partial\otimes \mathrm{id}_E)\bigl(a\,\overline\nabla_E(e)+\overline{\partial} a\otimes e\bigr) \notag\\
&=\ a\,\beta_E(e)+\partial a\wedge \overline\nabla_E(e)+\partial\overline{\partial} a\otimes e.
\label{eq:beta-expand}
\end{align}
Secondly, by the definition of $D_E$, we have
\begin{align}
D_E(\partial a\otimes e)
&\ =\ (\overline{\partial}\partial a)\otimes e-(\wedge\otimes \mathrm{id}_E)(\partial a\otimes \overline\nabla_E(e)) \notag\\
&=\ \overline{\partial}\partial a\otimes e-\partial a\wedge \overline\nabla_E(e).
\label{eq:D-expand}
\end{align}
Thirdly, again from \eqref{eq:DE-def}, we have the following:
\begin{align}
D_E(a\eta)\
&=\ (\overline{\partial}\otimes \mathrm{id}_E)(a\eta)
-(\wedge\otimes \mathrm{id}_E)(\mathrm{id}_{\Omega^{1,0}}\otimes \overline\nabla_E)(a\eta) \notag\\
&=\ \overline{\partial} a\wedge \eta+a\,D_E(\eta).
\label{eq:D-aeta}
\end{align}
Since $\mu$ is an isomorphism of left modules, it follows that $\widetilde{D}_E$ is a $\overline{\partial}$-connection on $\Omega^{1,0}\otimes_{\mathcal{A}} E$. Combining \eqref{eq:beta-expand}, \eqref{eq:D-expand} and \eqref{eq:D-aeta}, and using
\[
\partial\overline{\partial}+\overline{\partial}\partial\ =\ 0,
\]
we finally obtain the following:
\[
\beta_E(ae)+D_E(\partial a\otimes e+a\eta)\ =\
a\,\beta_E(e)+a\,D_E(\eta)+\overline{\partial} a\wedge \eta\,.
\]

Since $\overline{\partial} a\wedge \eta\,=\,(\mu\otimes \mathrm{id}_E)(\overline{\partial} a\otimes \eta)$, this is exactly \eqref{required1}. Hence \eqref{required} holds, and therefore, $\overline\nabla_J$ satisfies the left Leibniz rule.

Finally, for the case of a noncommutative complex curve $\mathcal{A}$, by definition we have $\Omega^{0,2}(\mathcal{A})\,=\,0$. Therefore, every left $\overline\partial$-connection on any f.g.p. left $\mathcal{A}$-module is automatically flat. Consequently, $\overline{\nabla}_J$ induces a holomorphic structure on $J_E^1$.
\end{proof}

We have the following short exact sequence
\begin{equation}\label{hol jet}
0\ \longrightarrow\ \Omega^{1,0}\otimes_{\mathcal A}E\
\overset{i}{\longrightarrow}\ J_E^1\
\overset{\pi}{\longrightarrow}\ E\
\longrightarrow\ 0
\end{equation}
of f.g.p. left \(\mathcal A\)-modules, where we equip \(J_E^1\) with the holomorphic structure constructed in \Pref{holomorphic structure}.

\begin{crlre}\label{short exact hol cat}
The sequence \eqref{hol jet} is short exact in the category of holomorphic vector bundles over $\mathcal{A}$.
\end{crlre}

\begin{proof}
We only need to verify compatibility with the holomorphic structures. By \Pref{holomorphic structure}, we have
\[
\overline{\nabla}_J(e,\,\eta)\ =\
\bigl(\overline{\nabla}_E(e),\,\widetilde{\beta}_E(e)+\widetilde{D}_E(\eta)\bigr).
\]
Therefore, the restriction of \(\overline{\nabla}_J\) to the submodule \(\Omega^{1,0}\otimes_{\mathcal A}E\,\subseteq \,J_E^1\) is exactly
\[
\widetilde{D}_E\ :\ \Omega^{1,0}\otimes_{\mathcal A}E\ \longrightarrow\
\Omega^{0,1}\otimes_{\mathcal A}\left(\Omega^{1,0}\otimes_{\mathcal A}E\right).
\]
Since \(\Omega^{0,2}(\mathcal{A})\,=\,0\) by definition, this induced \((0,\,1)\)-connection is automatically flat. Hence \(\Omega^{1,0}\otimes_{\mathcal A}E\) itself is a holomorphic vector bundle. Finally, it follows immediately that both maps in \eqref{hol jet} are holomorphic, because
\[
\overline{\nabla}_J\circ i=(\mathrm{id}\otimes i)\circ \widetilde{D}_E\,,
\qquad
(\mathrm{id}\otimes \pi)\circ \overline{\nabla}_J=\overline{\nabla}_E\circ \pi.
\]
This completes the proof.
\end{proof}

\begin{crlre}
The space of holomorphic sections of the jet bundle $J_E^1$ is given by $H^0(J_E^1\,,\overline{\nabla}_J)=H^0(E,\,\overline{\nabla}_E)\oplus_{\mathbb{C}}H^0(\Omega^{1,0}\otimes_{\mathcal{A}}E,\,\widetilde{D}_E)$.
\end{crlre}
\begin{prf}
Follows from \Dref{hol sections} and \Pref{holomorphic structure}.\qed
\end{prf}

\begin{ppsn}[{Functoriality}]\label{functor}
The assignment $(E,\,\overline{\nabla}_E)\,\longmapsto\,(J_E^1,\,\overline{\nabla}_J)$ is an endofunctor on the category of holomorphic vector bundles over a noncommutative complex curve $\mathcal{A}$. Moreover, if $\phi\,:\,(E,\,\overline{\nabla}_E)\, \longrightarrow\, (F,\,\overline{\nabla}_F)$ is a morphism of holomorphic vector bundles over $\mathcal{A}$, then there exists a commutative diagram of short exact sequences in the holomorphic category
\[
\begin{tikzcd}[row sep=small, column sep=tiny] 
0 &\longrightarrow& \Omega^{1,0}\otimes_{\mathcal{A}}E \arrow[d, "\,\mathrm{id}\otimes\phi"]  & \longrightarrow & J^1_E \arrow[d, "\,J^1\phi"]  & \longrightarrow & E \arrow[d, "\,\phi"] &\longrightarrow& 0\\ 
0 &\longrightarrow& \Omega^{1,0}\otimes_{\mathcal{A}}F & \longrightarrow & J^1_F & \longrightarrow & F &\longrightarrow& 0
\end{tikzcd}
\]
where all the vertical arrows are holomorphic.
\end{ppsn}
\begin{prf}
Consider a noncommutative complex curve $\mathcal{A}$. Recall that a morphism $\phi:(E,\overline{\nabla}_E)\longrightarrow(F,\overline{\nabla}_F)$ of holomorphic vector bundles over $\mathcal{A}$ is by definition a left $\mathcal{A}$-linear map from $E$ to $F$ such that $(\mathrm{id}\otimes\phi)\circ\overline{\nabla}_E=\overline{\nabla}_F\circ\phi$. Given such a map $\phi$, define
\[
J^1\phi:J_E^1\longrightarrow J_F^1\,,\quad (e,\eta)\longmapsto\big(\phi(e),\,(\mathrm{id}\otimes\phi)(\eta)\big)\,.
\]
It is easy to check that $J^1\phi$ is a left $\mathcal{A}$-linear map. Moreover, we get the following identities
\begin{IEEEeqnarray}{lCl}\label{identities}
& & (\mathrm{id}\otimes\mathrm{id}\otimes\phi)\widetilde\beta_E(e)= \widetilde\beta_F(\phi(e))\nonumber\\
& & (\mathrm{id}\otimes\mathrm{id}\otimes\phi)\widetilde{D}_E(\eta)= \widetilde{D}_F\big((\mathrm{id}\otimes\phi)(\eta)\big)
\end{IEEEeqnarray}
for all $e\in E$ and $\eta\in\Omega^{1,0}\otimes_{\mathcal{A}}E$. Let $\overline{\nabla}_J^{(E)}$ and $\overline{\nabla}_J^{(F)}$ denote the holomorphic structures obtained in \Pref{holomorphic structure} for the jet modules $J_E^1$ and $J_F^1$ respectively. Using the identities in \Eref{identities}, we obtain
\begin{IEEEeqnarray*}{lCl}
\overline{\nabla}_J^{(F)}(J^1\phi(e,\eta)) &=& \big(\overline{\nabla}_F(\phi(e)),\,\widetilde{\beta}_F(\phi(e))+\widetilde{D}_F((\mathrm{id}\otimes\phi)(\eta))\big)\\
&=& \big((\mathrm{id}\otimes\phi)\overline{\nabla}_E(e),\,(\mathrm{id}\otimes\mathrm{id}\otimes\phi)\widetilde{\beta}_E(e)+(\mathrm{id}\otimes\mathrm{id}\otimes\phi)\widetilde{D}_E(\eta)\big)\\
&=& (\mathrm{id}_{\Omega^{0,1}}\otimes J^1\phi)\overline{\nabla}_J^{(E)}(e,\eta)
\end{IEEEeqnarray*}
which proves that $J^1\phi$ is a holomorphic map. Finally, it is easy to observe that $J^1(\mathrm{id}_E)=\mathrm{id}_{J_E^1}$, and $J^1(\psi\circ\phi)=(J^1\psi)\circ (J^1\phi)$ for morphisms $\phi:(E,\overline{\nabla}_E)\longrightarrow(F,\overline{\nabla}_F)$ and $\psi:(F,\overline{\nabla}_F)\longrightarrow(G,\overline{\nabla}_G)$ of holomorphic vector bundles over $\mathcal{A}$. Thus the assignment $(E,\overline{\nabla}_E)\longmapsto(J_E^1\,,\overline{\nabla}_J)$ is an endofunctor on the category of holomorphic vector bundles over $\mathcal{A}$.

The final statement of diagram commutativity is easy to verify.\qed
\end{prf}

Note that the connection $\overline{\nabla}_J$ constructed in \Pref{holomorphic structure} does not involve any auxiliary choice beyond the given holomorphic structure $\overline{\nabla}_E$ on $E$ and the module structure on $J_E^1\,$. It is determined functorially from the data already present. If
\[
T\ :\ J_E^1\ \longrightarrow\ \Omega^{0,1}\otimes_{\mathcal A}J_E^1
\]
is any left $\mathcal A$-linear map, then
\[
\overline{\nabla}_J^{\,\prime}\ :=\ \overline{\nabla}_J+T
\]
is again a left $\overline\partial$-connection on $J_E^1$. Since $\Omega^{0,2}\,=\,0$ by the assumption on $\mathcal{A}$, every such left $\overline\partial$-connection $\overline{\nabla}_J^{\,\prime}$ on $J_E^1$ is automatically flat. Therefore, $J_E^1$ has many holomorphic structures. What distinguishes the connection $\overline{\nabla}_J$ from these other connections is its compatibility with the jet sequence in the holomorphic category that we shall see in the next section.


\section{Holomorphic connection and splitting of the jet sequence}\label{Sec4}

Let \((E,\,\overline{\nabla}_E)\) be a holomorphic vector bundle over a noncommutative complex curve $\mathcal{A}$. By Corollary \ref{short exact hol cat}, we have the following short exact sequence of holomorphic vector bundles
\[
0\longrightarrow\big(\Omega^{1,0}\otimes_{\mathcal A}E\,,\,\widetilde{D}_E\big)
\overset{i}{\longrightarrow} \left(J_E^1\,,\,\overline{\nabla}_J\right)
\overset{\pi}{\longrightarrow} \left(E\,,\,\overline{\nabla}_E\right)
\longrightarrow 0\,.
\]

\begin{dfn}
The extension class $At(E)\in \mathrm{Ext}_{\mathrm{hol}}^1(E,\,\Omega^{1,0}\otimes_{\mathcal{A}}E)$ associated to the holomorphic jet sequence is called the Atiyah class.
\end{dfn}

A natural question that arises is when $At(E)$ vanishes, which is equivalent to splitting of the holomorphic jet sequence in the holomorphic category. In the classical case, $At(E)$ vanishes precisely when $E$ admits a holomorphic connection. In this section, we shall take this to the noncommutative setting of the noncommutative complex curves.


\subsection{Holomorphic connection and holomorphic splitting}\label{Sec4.1}

Recall that given a differential calculus $(\Omega^\bullet,d)$ on $\mathcal{A}$, if $\nabla\,:\,E\,\longrightarrow\,\Omega^1\otimes_{\mathcal{A}}E$ is a (left) $d$-connection on a f.g.p. left $\mathcal{A}$-module $E$, the curvature is given by $\Theta_\nabla\,:=\,\widetilde{\nabla}\circ\nabla$, where
\[
\widetilde{\nabla}\ :\ \Omega^1\otimes_{\mathcal{A}}E\ \longrightarrow\ \Omega^2\otimes_{\mathcal{A}}E
\]
is defined by
\begin{IEEEeqnarray}{lCl}\label{curvature}
\omega\otimes e\ \longmapsto\ d\omega\otimes e-\omega\wedge\nabla e
\end{IEEEeqnarray}
for all $\omega\,\in\,\Omega^1$ and $e\,\in\, E$ (see \cite{Con94}, or Section 3.2 in \cite{BM20}). It is easy to verify that $\Theta_\nabla\in\mathrm{Hom}_{\mathcal{A}}(E,\,\Omega^2\otimes_{\mathcal{A}} E)$.

\begin{dfn}\label{hol con}
A holomorphic connection on a holomorphic vector bundle \((E,\,\overline{\nabla}_E)\) over a noncommutative complex curve $\mathcal{A}$ is a left \(\partial\)-connection
\[
\nabla^{1,0}\ :\ E\ \longrightarrow\ \Omega^{1,0}\otimes_{\mathcal A}E
\]
such that the total \(d\,:=\,\partial+\overline\partial\)-connection
\[
\nabla_{\mathrm{tot}}\ :=\ \nabla^{1,0}+\overline{\nabla}_E\
:\ E\ \longrightarrow\ \Omega^1\otimes_{\mathcal A}E
\]
has zero curvature.
\end{dfn}

\begin{lmma}\label{curvature expression}
Given a holomorphic vector bundle \((E,\,\overline{\nabla}_E)\) over a noncommutative complex curve $\mathcal{A}$ and a $\partial$-connection $\nabla^{1,0}\,:\,E\,\longrightarrow\, \Omega^{1,0}\otimes_{\mathcal A}E$, the curvature of the total connection $\nabla_{\mathrm{tot}}=\nabla^{1,0}+\overline{\nabla}_E$ is given by the following:
\[
\Theta_{\nabla_{\mathrm{tot}}}
=
(\partial\otimes\mathrm{id}_E)\overline{\nabla}_E
+(\overline\partial\otimes\mathrm{id}_E)\nabla^{1,0}
-(\wedge\otimes\mathrm{id}_E)(\mathrm{id}_{\Omega^{1,0}}\otimes\overline{\nabla}_E)\nabla^{1,0}
-(\mu\otimes\mathrm{id}_E)(\mathrm{id}_{\Omega^{0,1}}\otimes\nabla^{1,0})\overline{\nabla}_E
\]
where $\mu$ is as in \eqref{mu}.
\end{lmma}

\begin{proof}
Consider the connection $\nabla_{\mathrm{tot}}\,:=\,\nabla^{1,0}+\overline{\nabla}_E\,:\,E\,\longrightarrow\,\Omega^1\otimes_{\mathcal{A}}E$, where $\Omega^1\,=\,\Omega^{1,0}\oplus\Omega^{0,1}$. Its curvature is $\widetilde{\nabla}_{\mathrm{tot}}\circ\nabla_{\mathrm{tot}}$, where
\[
\widetilde{\nabla}_{\mathrm{tot}}\ :\ \Omega^1\otimes_{\mathcal{A}}E\ \longrightarrow\ \Omega^2\otimes_{\mathcal{A}}E\ =\ \Omega^{1,1}\otimes_{\mathcal{A}}E
\]
is given by
\[
\omega\otimes e\ \longmapsto\ d\omega\otimes e-\omega\wedge\nabla_{\mathrm{tot}}\,(e)
\]
for all $\omega\,\in\, \Omega^1$ and $e\,\in\, E$ (see \eqref{curvature}). Since $d\,=\,\partial+\overline\partial$, we conclude that the map
\[
\widetilde{\nabla}_{\mathrm{tot}}\ :\ \Omega^{1,0}\otimes_{\mathcal{A}}E\oplus\Omega^{0,1}\otimes_{\mathcal{A}}E\ \longrightarrow\ \Omega^{1,1}\otimes_{\mathcal{A}}E
\]
is given by
\[
(\omega_1\otimes e_1,\,\omega_2\otimes e_2)\ \longmapsto\ \overline\partial\omega_1\otimes e_1+\partial\omega_2\otimes e_2-\omega_1\wedge\nabla_{\mathrm{tot}}\,(e_1)-\omega_2\wedge\nabla_{\mathrm{tot}}\,(e_2).
\]
Now writing $\nabla^{1,0}(e)\,=\,\omega_1\otimes e_1$ and $\overline{\nabla}_E\,(e)\,=\,\omega_2\otimes e_2$ with $\omega_1\,\in\,\Omega^{1,0}$ and $\omega_2\,\in\,\Omega^{0,1}$ (we can safely ignore the summation involved here in our computation below), we get that
\begin{IEEEeqnarray}{lCl}\label{used here}
\widetilde{\nabla}_{\mathrm{tot}}\circ\nabla_{\mathrm{tot}}(e) &=& \widetilde{\nabla}_{\mathrm{tot}}(\omega_1\otimes e_1\,,\,\omega_2\otimes e_2)\nonumber\\
&=& \overline\partial\omega_1\otimes e_1+\partial\omega_2\otimes e_2-\omega_1\wedge\nabla_{\mathrm{tot}}\,(e_1)-\omega_2\wedge\nabla_{\mathrm{tot}}\,(e_2)\nonumber\\
&=& \overline\partial\omega_1\otimes e_1+\partial\omega_2\otimes e_2-\omega_1\wedge\overline{\nabla}_E\,(e_1)-\omega_2\wedge\nabla^{1,0}\,(e_2)
\end{IEEEeqnarray}
because $\omega_1\wedge\nabla^{1,0}\,e_1\,=\,\omega_2\wedge\overline{\nabla}_E\,e_2\,=\,0$ as $\Omega^{2,0}\,=\,\Omega^{0,2}\,=\,\{0\}$. Now, observe that
\begin{align*}
\omega_1\wedge\overline{\nabla}_E\,(e_1) &= (\wedge\otimes\mathrm{id}_E)(\mathrm{id}_{\Omega^{1,0}}\otimes\overline{\nabla}_E)\circ\nabla^{1,0}(e),\\
\omega_2\wedge\nabla^{1,0}\,(e_2) &= (\mu\otimes\mathrm{id}_E)(\mathrm{id}_{\Omega^{0,1}}\otimes\nabla^{1,0})\circ\overline{\nabla}_E(e).
\end{align*}
Hence, from \eqref{used here} we get that
\begin{IEEEeqnarray}{lCl}\label{curvature expression-eq}
& & \widetilde{\nabla}_{\mathrm{tot}}\circ\nabla_{\mathrm{tot}}(e)\\
&=& (\partial\otimes\mathrm{id}_E)\overline{\nabla}_E(e)+(\overline\partial\otimes\mathrm{id}_E)\nabla^{1,0}(e)-(\wedge\otimes\mathrm{id}_E)(\mathrm{id}_{\Omega^{1,0}}\otimes\overline{\nabla}_E)\circ\nabla^{1,0}(e)\nonumber\\
& & -(\mu\otimes\mathrm{id}_E)(\mathrm{id}_{\Omega^{0,1}}\otimes\nabla^{1,0})\circ\overline{\nabla}_E(e).\nonumber
\end{IEEEeqnarray}
for all $e\,\in\, E$. This completes the proof.
\end{proof}

\begin{thm}\label{thm:splitimpliesconnectionbothway}
Let \((E,\,\overline{\nabla}_E)\) be a holomorphic vector bundle over a noncommutative complex curve. Then \(E\) admits a holomorphic connection if and only if the holomorphic jet sequence in \eqref{hol jet} splits holomorphically.
\end{thm}

\begin{proof}
Let $\mathcal{A}$ be a noncommutative complex curve and \((E,\,\overline{\nabla}_E)\) be a holomorphic vector bundle over it. Assume first that \(\nabla^{1,0}\) is a holomorphic connection on \(E\). Define
\[
s_{\nabla}\ :\ E\ \longrightarrow\ J_E^1,
\qquad
s_{\nabla}(e)\ :=\ \bigl(e,\,\nabla^{1,0}e\bigr).
\]
Since \(\nabla^{1,0}\) is a left \(\partial\)-connection, the defining left action on \(J_E^1\) gives
\[
s_{\nabla}(ae)
=
\bigl(ae,\nabla^{1,0}(ae)\bigr)
=
\bigl(ae,\,a\,\nabla^{1,0}(e)+\partial a\otimes e\bigr)
=
a\cdot s_{\nabla}(e)
\]
for all $a\in\mathcal{A}$, so that \(s_{\nabla}\) is left \(\mathcal A\)-linear and satisfies
the condition \(\pi\circ s_{\nabla}\,=\,\mathrm{id}_E\).

To verify that \(s_{\nabla}\) is holomorphic, consider the defect
\[
\delta_{\nabla}\ :=\ \overline{\nabla}_J\circ s_{\nabla}
- (\mathrm{id}\otimes s_{\nabla})\circ \overline{\nabla}_E
\ :\ E\ \longrightarrow \Omega^{0,1}\otimes_{\mathcal A}J_E^1.
\]
Using \Pref{holomorphic structure}, we find that
\[
\overline{\nabla}_J\bigl(s_{\nabla}(e)\bigr)
=
\Bigl(\overline{\nabla}_E(e),\,\widetilde{\beta}_E(e)+\widetilde{D}_E(\nabla^{1,0}e)\Bigr),
\]
while
\[
(\mathrm{id}\otimes s_{\nabla})\overline{\nabla}_E(e)
=
\Bigl(\overline{\nabla}_E(e),\,(\mathrm{id}\otimes\nabla^{1,0})\overline{\nabla}_E(e)\Bigr).
\]
Thus, the first component of \(\delta_{\nabla}\) vanishes, and its second component is
\[
\widetilde{\beta}_E(e)
+\widetilde{D}_E(\nabla^{1,0}e)
-(\mathrm{id}\otimes\nabla^{1,0})\overline{\nabla}_E(e).
\]
Applying \(\mu\otimes\mathrm{id}_E\) and using the definitions of \(\widetilde{\beta}_E\) and \(\widetilde{D}_E\) from \Pref{holomorphic structure}, we precisely recover the right-hand side of \eqref{curvature expression-eq}. Since
\[
\mu\ :\ \Omega^{0,1}\otimes_{\mathcal A}\Omega^{1,0}\ \longrightarrow\ \Omega^{1,1}
\]
is an isomorphism, we conclude the following:
$$\delta_{\nabla}\ =\ 0$$ if and only if the curvature satisfies $$\Theta_{\nabla_{\mathrm{tot}}}\ =\ 0.$$
As \(\nabla^{1,0}\) is a holomorphic connection, by definition
\(\Theta_{\nabla_{\mathrm{tot}}}=0\). Hence \(\delta_{\nabla}=0\), so \(s_{\nabla}\) is a holomorphic splitting of the jet sequence in \eqref{hol jet}.
\smallskip

Conversely, suppose that \(s\,:\,E\,\longrightarrow\, J_E^1\) is a holomorphic splitting of \eqref{hol jet}.
Since \(\pi\circ s\,=\,\mathrm{id}_E\) and also \(J_E^1\,=\,E\oplus_{\mathbb C}(\Omega^{1,0}\otimes_{\mathcal A}E)\) as a vector space, there is a unique \(\mathbb C\)-linear map
\[
\nabla^{1,0}\ :\ E\ \longrightarrow\ \Omega^{1,0}\otimes_{\mathcal A}E
\]
such that
\[
s(e)\ =\ \bigl(e,\,\nabla^{1,0}e\bigr).
\]
Since \(s\) is left \(\mathcal A\)-linear, the left action on \(J_E^1\) implies that
\[
\nabla^{1,0}(ae)\ =\ a\,\nabla^{1,0}(e)+\partial a\otimes e,
\]
for all $a\in\mathcal{A}$ and $e\in E$. So \(\nabla^{1,0}\) is a left \(\partial\)-connection on $E$. Moreover, \(s\) is holomorphic, and hence
\[
\overline{\nabla}_J\circ s\ =\ (\mathrm{id}\otimes s)\circ\overline{\nabla}_E.
\]
Repeating the above computation with \(s\,=\,s_{\nabla}\), we obtain that \(\delta_{\nabla}\,=\,0\), and therefore we have \(\Theta_{\nabla_{\mathrm{tot}}}\,=\,0\) by \eqref{curvature expression-eq}. Thus, \(\nabla^{1,0}\) is a holomorphic connection over \(E\).
\end{proof}

A holomorphic direct sum decomposition of a holomorphic vector bundle $(E,\overline{\nabla}_E)$ over $\mathcal{A}$ is a decomposition $E\,=\,\bigoplus_{{j=1}}^nE_j$ of f.g.p. left $\mathcal{A}$-submodules such that each $E_j$ is a holomorphic subbundle (see Definition \ref{holomorphic subbundle})\,:
\[
\overline{\nabla}_E(E_j)\ \subseteq\ \Omega^{0,1}\otimes_{\mathcal{A}}E_j
\]
for all $j\,=\,1,\,\cdots,\,n$.

\begin{ppsn}[Direct sum]
Suppose that a holomorphic vector bundle $E$ over a noncommutative complex curve decomposes into holomorphic subbundles $E\,=\,\bigoplus_{{j=1}}^n\,E_j$. Then, $E$ admits a holomorphic connection if and only if each $E_j$ admits holomorphic connection.
\end{ppsn}

\begin{proof}
If each $E_j$ admits a holomorphic connection $\nabla_j$, then $\bigoplus_{j=1}^n \nabla_j$ is
a holomorphic connection on $E\,=\,\bigoplus_{{j=1}}^n\,E_j$. To prove the converse, let
$\nabla$ be a holomorphic connection on $E\,=\,\bigoplus_{{j=1}}^n\,E_j$. For
each $1\ \leq\, i\, \leq\, n$, denote by $\nabla_i$ the
following composition
$$
E_i \, \hookrightarrow\, E\,=\,\bigoplus_{{j=1}}^n\,E_j \, \stackrel{\nabla}{\longrightarrow}\,
\Omega^{1,0}\otimes_{\mathcal A} E \, \longrightarrow\, \Omega^{1,0}\otimes_{\mathcal A} E_i,
$$
where the projection $\Omega^{1,0}\otimes_{\mathcal A} E \, \longrightarrow\,
\Omega^{1,0}\otimes_{\mathcal A} E_i$ is the tensor product of the identity map of $\Omega^{1,0}$
and the natural projection $\bigoplus_{{j=1}}^n\,E_j \,\longrightarrow\, E_i$. It is straightforward
to check that $\nabla_i$ is a holomorphic connection on $E_i$.
\end{proof}

\noindent\textbf{An instance of Holomorphic connection}: For the standard holomorphic structure on the line bundle \(\mathcal L_n\) over the quantum projective line $\mathbb{CP}_q^1\,$, the canonical \((1,0)\)-connection
\[
\nabla^{\partial}(\xi):=q^{-n-2}\,\omega_+\bigl(X_+\triangleright \xi\bigr)
\]
constructed in \cite[Eq.~(3.30)]{KLvS11} has the curvature
\[
\Theta_{\nabla^{\partial}+\overline{\nabla}^{(n)}_0}
=
-q^{-n+1}[n]_q\,\omega_-\wedge\omega_+,
\]
see \cite[Eq.~(3.29)]{KLvS11}. Hence, this canonical $\partial$-connection on $\mathcal{L}_n$ is a holomorphic connection if and only if \(n=0\).


\subsection{The bimodule case}\label{Sec4.2}

Assume that \(E\) is additionally a \(\mathcal A\)-bimodule and the given holomorphic structure
\[
\overline{\nabla}_E:E\longrightarrow \Omega^{0,1}\otimes_{\mathcal A}E
\]
on $E$ is a bimodule \(\overline{\sigma}_E\)-connection, where $\overline{\sigma}_E:E\otimes_{\mathcal A}\Omega^{0,1}\to\Omega^{0,1}\otimes_{\mathcal A}E$ is a $\mathcal{A}$-bimodule map.

\begin{lmma}
If a holomorphic vector bundle $(E,\,\overline{\nabla}_E)$ over $\mathcal{A}$ admits a holomorphic connection which is further a bimodule connection, then the jet sequence splits both in the holomorphic as well as in the bimodule category.
\end{lmma}

\begin{proof}
This follows from the combination of \Tref{thm:splitimpliesconnectionbothway} and \Lref{bimod splitting}.
\end{proof}

\begin{lmma}
Let $(E,\,\overline{\nabla}_E)$ be a holomorphic vector bundle over $\mathcal{A}$ which admits a holomorphic connection and \(\sigma_E\,:\,E\otimes_{\mathcal A}\Omega^{1,0}\,\longrightarrow\,\Omega^{1,0}\otimes_{\mathcal A}E\) be a $\mathcal{A}$-bimodule map. Then it is a $\sigma_E$-bimodule connection if and only if the total $d$-connection is a $\,\overline{\sigma}_E\oplus\sigma_E$-bimodule connection.
\end{lmma}

\begin{proof}
Let \(\sigma_E\,:\,E\otimes_{\mathcal A}\Omega^{1,0}\,\longrightarrow\,\Omega^{1,0}\otimes_{\mathcal A}E\) be a bimodule map, and consider
\[
\sigma\ :\ E\otimes_{\mathcal A}\Omega^1\ \longrightarrow\ \Omega^1\otimes_{\mathcal A}E,
\ \ \, e\otimes(\omega^{0,1}+\omega^{1,0})\, \longmapsto
\,\overline{\sigma}_E(e\otimes\omega^{0,1})+\sigma_E(e\otimes\omega^{1,0}).
\]
We show that a left $\partial$-connection \(\nabla^{1,0}\) on $E$ is a right \(\sigma_E\)-connection if and only if the total $d\,=\,\partial+\overline\partial$-connection
\(\nabla_{\mathrm{tot}}\,=\,\nabla^{1,0}+\overline{\nabla}_E\) is a right \(\sigma\)-connection.

Assume first that \(\nabla^{1,0}\) is a right \(\sigma_E\)-connection.
For \(e\,\in \,E\) and \(a\,\in\,\mathcal A\),
\begin{align*}
\nabla_{\mathrm{tot}}(ea)
&=\nabla^{1,0}(ea)+\overline{\nabla}_E(ea)\\
&=\nabla^{1,0}(e)a+\sigma_E(e\otimes\partial a)
+\overline{\nabla}_E(e)a+\overline{\sigma}_E(e\otimes\overline\partial a)\\
&=\nabla_{\mathrm{tot}}(e)a+\sigma(e\otimes da).
\end{align*}
Thus, \(\nabla_{\mathrm{tot}}\) is a right \(\sigma\)-connection. Conversely, suppose that \(\nabla_{\mathrm{tot}}\) is a right \(\sigma\)-connection. Using the right Leibniz rule for \(\overline{\nabla}_E\), we obtain the following\,:
\begin{align*}
\nabla^{1,0}(ea)
&=\nabla_{\mathrm{tot}}(ea)-\overline{\nabla}_E(ea)\\
&=\nabla_{\mathrm{tot}}(e)a+\sigma(e\otimes da)-\overline{\nabla}_E(e)a-\overline{\sigma}_E(e\otimes\overline\partial a)\\
&=\nabla^{1,0}(e)a+\sigma_E(e\otimes\partial a),
\end{align*}
which shows that \(\nabla^{1,0}\) is a right \(\sigma_E\)-connection.
\end{proof}

Given a right \(\sigma_E\)-connection \(\nabla^{1,0}\) on $E$, define the $\mathbb{C}$-linear map
\[
D_E^{1,0}\ :\ \Omega^{0,1}\otimes_{\mathcal A}E\ \longrightarrow\
\Omega^{1,1}\otimes_{\mathcal A}E
\]
by
\[
D_E^{1,0}\ :=\ -\partial\otimes\mathrm{id}_E
+
(\mu\otimes\mathrm{id}_E)\circ(\mathrm{id}_{\Omega^{0,1}}\otimes\nabla^{1,0})
\]
where $\mu$ is as in \Eref{mu}. Suppose that there exists
\[
\Phi_{(2)}\ :\ 
\Omega^{0,1}\otimes_{\mathcal A}\Omega^{1,0}\ \longrightarrow\
\Omega^{1,0}\otimes_{\mathcal A}\Omega^{0,1}
\]
which is a $\mathcal{A}$-bimodule isomorphism (for example, when the calculus over $\mathcal{A}$ is factorisable). Define the $\mathbb{C}$-linear map
\[
\Psi_{\nabla}
:=
(\Phi_{(2)}\circ\mu^{-1}\otimes\mathrm{id}_E)\circ D_E^{1,0}\,:\,
\Omega^{0,1}\otimes_{\mathcal A}E
\longrightarrow
\Omega^{1,0}\otimes_{\mathcal A}(\Omega^{0,1}\otimes_{\mathcal A}E).
\]
Observe that the following map $$(\Phi_{(2)}\otimes\mathrm{id}_E)\circ(\mathrm{id}\otimes\sigma_E):\big(\Omega^{0,1}\otimes_{\mathcal A}E\big)\otimes_{\mathcal{A}}\Omega^{1,0}\longrightarrow\Omega^{1,0}\otimes_{\mathcal{A}}\big(\Omega^{0,1}\otimes_{\mathcal A}E\big)$$ is a $\mathcal{A}$-bimodule map.

\begin{lmma}
The map \(\Psi_{\nabla}\) is a right
\((\Phi_{(2)}\otimes\mathrm{id}_E)\circ(\mathrm{id}\otimes\sigma_E)\)-connection on
\(\Omega^{0,1}\otimes_{\mathcal A}E\).
\end{lmma}

\begin{proof}
Let \(\omega\otimes e\in \Omega^{0,1}\otimes_{\mathcal A}E\) and \(b\in\mathcal A\).
Since \(\nabla^{1,0}\) is a right \(\sigma_E\)-connection,
\begin{align*}
D_E^{1,0}\bigl((\omega\otimes e)b\bigr)
&=
-\partial\omega\otimes eb
+(\mu\otimes\mathrm{id}_E)\bigl(\omega\otimes\nabla^{1,0}(eb)\bigr)\\
&=
-\partial\omega\otimes eb
+(\mu\otimes\mathrm{id}_E)\bigl(\omega\otimes\nabla^{1,0}(e)b\bigr)
+(\mu\otimes\mathrm{id}_E)\circ(\mathrm{id}\otimes\sigma_E)(\omega\otimes e\otimes\partial b)\\
&=
D_E^{1,0}(\omega\otimes e)\,b
+
(\mu\otimes\mathrm{id}_E)\circ(\mathrm{id}\otimes\sigma_E)(\omega\otimes e\otimes\partial b).
\end{align*}
Applying \(\Phi_{(2)}\circ\mu^{-1}\otimes\mathrm{id}_E\), we conclude that
\[
\Psi_{\nabla}\bigl((\omega\otimes e)b\bigr)
=
\Psi_{\nabla}(\omega\otimes e)\,b
+
\bigl((\Phi_{(2)}\otimes\mathrm{id}_E)\circ(\mathrm{id}\otimes\sigma_E)\bigr)(\omega\otimes e\otimes\partial b),
\]
which is precisely the right Leibniz rule.
\end{proof}

From \Pref{holomorphic structure}, recall the $\overline\partial$-connection
\[
\widetilde{D}_E
=
(\mu^{-1}\otimes\mathrm{id}_E)\circ D_E\,,
\qquad
D_E
=
\overline\partial\otimes\mathrm{id}_E
-
(\wedge\otimes\mathrm{id}_E)(\mathrm{id}_{\Omega^{1,0}}\otimes\overline{\nabla}_E).
\]
on $\Omega^{1,0}\otimes_{\mathcal{A}}E$.

\begin{ppsn}
Let \((E,\,\overline{\nabla}_E)\) be a holomorphic vector bundle over $\mathcal{A}$ such that \(\overline{\nabla}_E\) is a right \(\overline{\sigma}_E\)-connection. If $E$ admits a holomorphic connection \(\nabla^{1,0}\) which is also a right \(\sigma_E\)-connection, then the following diagram commutes:
\[
\begin{tikzcd}
E\otimes\mathcal{A} \arrow[r, "\mathrm{id}\otimes\overline\partial"] \arrow[dr, "\mathrm{id}\otimes\partial" '] & E\otimes_{\mathcal{A}}\Omega^{0,1} \arrow[r, "\overline{\sigma}_E"] & \Omega^{0,1}\otimes_{\mathcal{A}}E \arrow[r, "\Psi_\nabla"] & \Omega^{1,0}\otimes_{\mathcal{A}}\Omega^{0,1}\otimes_{\mathcal{A}}E \\
& E\otimes_{\mathcal{A}}\Omega^{1,0} \arrow[r, "\sigma_E"] & \Omega^{1,0}\otimes_{\mathcal{A}}E \arrow[r, "\widetilde D_E"] & \Omega^{0,1}\otimes_{\mathcal{A}}\Omega^{1,0}\otimes_{\mathcal{A}}E. \arrow[u, "\Phi_{(2)}\otimes\mathrm{id}" ']
\end{tikzcd}
\]
\end{ppsn}

\begin{proof}
If $E$ admits a holomorphic connection \(\nabla^{1,0}\) on it, then the curvature $\Theta_{\nabla_{\mathrm{tot}}}$ of the total connection $\nabla_{\mathrm{tot}}=\nabla^{1,0}+\overline{\nabla}_E$ is zero by definition. For \(e\in E\) and \(a\in\mathcal A\), the vanishing of the curvature
\(\Theta_{\nabla_{\mathrm{tot}}}(ea)=0\) and $\Theta_{\nabla_{\mathrm{tot}}}(e)=0$, together with \eqref{curvature expression-eq}, yields
\begin{align*}
0
&=
\Theta_{\nabla_{\mathrm{tot}}}(ea)-\Theta_{\nabla_{\mathrm{tot}}}(e)a\\
&=
-D_E^{1,0}\bigl(\overline{\sigma}_E(e\otimes\overline\partial a)\bigr)
+
D_E\bigl(\sigma_E(e\otimes\partial a)\bigr).
\end{align*}
Applying \(\Phi_{(2)}\circ\mu^{-1}\otimes\mathrm{id}_E\) to both sides, we conclude that
\[
\Psi_\nabla\bigl(\overline{\sigma}_E(e\otimes\overline\partial a)\bigr)
=
(\Phi_{(2)}\otimes\mathrm{id}_E)\widetilde D_E\bigl(\sigma_E(e\otimes\partial a)\bigr),
\]
which is precisely the commutativity of the diagram.
\end{proof}


\section{The case of the quantum projective line}\label{sec5}

The quantum projective line $\mathbb{CP}_q^1$ exhibits a special feature among the noncommutative complex curves. Classically, this is the Riemann sphere. We refer to \cite{KLvS11} for essential facts.

\subsection{Lift of holomorphic structure as bimodule connection}

Recall from Definition \ref{bimod con} the definition of a bimodule connection. The morphism between two bimodule connections $(E,\,\nabla_E,\,\sigma_E)$ and $(F,\,\nabla_F,\,\sigma_F)$ over $E$ and $F$ respectively, is a bimodule map $\theta\,:\,E\,\longrightarrow\, F$ that intertwines the bimodule covariant derivatives \cite[Section 3.4.2]{BM20}. The category of bimodule connections is rich and, in fact, is a monoidal category \cite[Theorem 3.78]{BM20}. It may be noted that being a bimodule connection is a property of a left connection, and not an additional independent data.

Let $\mathcal{A}(\mathbb{CP}_q^1)$ denote the coordinate algebra of the quantum projective line. We simply denote it by $\mathcal{A}$ for notational brevity throughout this section.
\medskip

Suppose that $E$ is a holomorphic vector bundle over $\mathcal{A}$, and in addition, $E$ is $\mathcal{A}$-bimodule. Also suppose that the given holomorphic structure $\overline{\nabla}_E$ on $E$ is, in particular, a bimodule connection. This holds, for example, for line bundles $\mathcal L_n$ over $\mathbb{CP}_q^1$ equipped with the standard holomorphic structures \cite{KLvS11}. We investigate whether the holomorphic structure $\overline{\nabla}_J$ on $J_E^1\,$, obtained in \Pref{holomorphic structure}, is again a bimodule connection in the case of $\mathbb{CP}_q^1\,$.

Assume that there exists a (unique) bimodule map
\[
\overline{\sigma}_E\ :\ E\otimes_{\mathcal A}\Omega^{0,1}\ \longrightarrow\ \Omega^{0,1}\otimes_{\mathcal A}E
\]
such that
\[
\overline{\nabla}_E(ea)\ =\ \overline{\nabla}_E(e)a+\overline{\sigma}_E(e\otimes \overline{\partial} a)
\]
for all $e\, \in\,  E$ and $a\,\in\, \mathcal A$. Recall from \Cref{Sec3.1} that, in order to define a right $\mathcal A$-module structure on $J_E^1\,$, one needs a bimodule map
\[
\sigma_E\ :\ E\otimes_{\mathcal A}\Omega^{1,0}\ \longrightarrow\ \Omega^{1,0}\otimes_{\mathcal A}E.
\]
A priori, there is no relationship between $\sigma_E$ and $\overline{\sigma}_E$. Thus, if one wants to lift the $\overline{\sigma}_E$-bimodule $\overline\partial$-connection $\overline{\nabla}_E$ to a bimodule connection on $J_E^1$ whose right $\mathcal A$-action is induced by $\sigma_E$, some sort of compatibility relations between $\sigma_E$ and $\overline{\sigma}_E$ is naturally required.

Also, recall from \Cref{Sec3.1} that if the jet sequence splits as $\mathcal{A}$-bimodules, then $E$ carries a bimodule $\partial$-connection $(\nabla_E,\,\sigma_E)$, i.e.,
\[
\nabla_E(ea)\ =\ \nabla_E(e)a+\sigma_E(e\otimes\partial a)
\]
and the converse also holds.
%

We begin by proving an extension result.
\begin{ppsn}\label{pp12}
Let $(E,\,\overline{\nabla}_E)$ be a holomorphic vector bundle over the quantum projective line $\mathcal{A}:=\mathcal{A}(\mathbb{CP}_q^1)$ such that $E$ is $\mathcal{A}$-bimodule and $\overline{\nabla}_E$ is a bimodule $\overline\partial$-connection. Assume that $E$ carries a bimodule $\partial$-connection so that the jet sequence in \eqref{ejs} splits in the bimodule category. Then any bimodule map
\[
\psi\ :\ \bigl(\Omega^{1,0}\otimes_{\mathcal{A}}E\bigr)\otimes_{\mathcal{A}}\Omega^{0,1}\ \longrightarrow\
\Omega^{0,1}\otimes_{\mathcal{A}}\bigl(\Omega^{1,0}\otimes_{\mathcal{A}}E\bigr)
\]
defined on the kernel part of the jet sequence extends to a bimodule map
\[
\widetilde{\psi}:J_E^1\otimes_{\mathcal{A}}\Omega^{0,1}
\ \longrightarrow\
\Omega^{0,1}\otimes_{\mathcal{A}}J_E^1
\]
on the jet bimodule $J_E^1\,$.
\end{ppsn}

\begin{proof}
A bimodule splitting of the jet sequence is equivalent to the choice of a bimodule $\partial$-connection on $E$. Let $\overline{\nabla}_E$ be a bimodule $\overline\partial$-connection with respect to bimodule map
\[
\overline{\sigma}_E\ :\ E\otimes_{\mathcal A}\Omega^{0,1}\ \longrightarrow\ \Omega^{0,1}\otimes_{\mathcal A}E,
\]
and let $(\nabla_E,\,\sigma_E)$ be a bimodule $\partial$-connection with respect to a bimodule map
\[
\sigma_E\ :\ E\otimes_{\mathcal A}\Omega^{1,0}\ \longrightarrow\ \Omega^{1,0}\otimes_{\mathcal A}E\,.
\]
Now, given $\psi$, define the following $\mathbb C$-linear map
\begin{align*}
\widetilde{\psi}\ :\ J_E^1\otimes_{\mathcal A}\Omega^{0,1}
&\longrightarrow
\Omega^{0,1}\otimes_{\mathcal A}J_E^1\\
e\otimes\omega\oplus\xi\,
&\longmapsto\,
\overline{\sigma}_E(e\otimes\omega)+\psi(\xi)
+\bigl((\mathrm{id}\otimes\nabla_E)\circ\overline{\sigma}_E-\psi\circ(\nabla_E\otimes\mathrm{id})\bigr)(e\otimes\omega).
\end{align*}
For $a\,\in\,\mathcal A$, using the fact that both $\overline{\sigma}_E$ and $\psi$ are left module maps, it follows that
\begin{IEEEeqnarray*}{lCl}
& & \widetilde{\psi}\bigl(a\,.\,(e\otimes\omega\oplus\xi)\bigr)\\
&=& \widetilde{\psi}(ae\otimes\omega\oplus a\xi+\partial a\otimes e\otimes\omega)\\
&=& \overline{\sigma}_E(ae\otimes\omega)\oplus\psi(a\xi)+\psi(\partial a\otimes e\otimes\omega)
+(\mathrm{id}\otimes\nabla_E)\circ\sigma_1(ae\otimes\omega)\\
& & -\psi\circ(\nabla_E\otimes\mathrm{id})(ae\otimes\omega)\\
&=& a\,\overline{\sigma}_E(e\otimes\omega)\oplus a\,\psi(\xi)+\psi(\partial a\otimes e\otimes\omega)
+a\,(\mathrm{id}\otimes\nabla_E)\circ\sigma_1(e\otimes\omega)\\
& & -\psi\bigl(a\nabla_E(e)\otimes\omega+\partial a\otimes e\otimes\omega\bigr)\\
&=& a\,\overline{\sigma}_E(e\otimes\omega)\oplus a\,\psi(\xi)
+a\,(\mathrm{id}\otimes\nabla_E)\circ\sigma_1(e\otimes\omega)
-a\,\psi\circ(\nabla_E\otimes\mathrm{id})(e\otimes\omega)\\
&=& a\,.\,\widetilde{\psi}(e\otimes\omega\oplus\xi).
\end{IEEEeqnarray*}
Hence $\widetilde{\psi}$ is a left $\mathcal A$-module map. Next we check right $\mathcal A$-linearity:
\begin{IEEEeqnarray*}{lCl}
& & \widetilde{\psi}\bigl((e\otimes\omega\oplus\xi)\,.\,a\bigr)\\
&=& \widetilde{\psi}(e\otimes\omega a\oplus \xi.a)\\
&=& \overline{\sigma}_E(e\otimes\omega a)\oplus\psi(\xi.a)
+(\mathrm{id}\otimes\nabla_E)\circ\overline{\sigma}_E(e\otimes\omega a)
-\psi\circ(\nabla_E\otimes\mathrm{id})(e\otimes\omega a)\\
&=& \overline{\sigma}_E(e\otimes\omega)a\oplus\psi(\xi)a
+(\mathrm{id}\otimes\nabla_E)\bigl(\overline{\sigma}_E(e\otimes\omega)a\bigr)
-\psi\circ(\nabla_E\otimes\mathrm{id})(e\otimes\omega a)\\
&=& \overline{\sigma}_E(e\otimes\omega)a\oplus\psi(\xi)a
+(\mathrm{id}\otimes\nabla_E)\bigl(\overline{\sigma}_E(e\otimes\omega)\bigr)a
+(\mathrm{id}\otimes\sigma_E)\bigl(\overline{\sigma}_E(e\otimes\omega)\otimes\partial a\bigr)\\
& & -\psi\circ(\nabla_E\otimes\mathrm{id})(e\otimes\omega)a.
\end{IEEEeqnarray*}
Here the last equality uses the twisted right Leibniz rule for $\nabla_E$ with respect to $\sigma_E$. Now, by the right multiplication on $\Omega^{0,1}\otimes_{\mathcal A}J_E^1$, and using the fact that
\[
\overline{\sigma}_E(e\otimes\omega)\ \in\ \Omega^{0,1}\otimes_{\mathcal A}E
\ \subseteq\ \Omega^{0,1}\otimes_{\mathcal A}J_E^1,
\]
we have
\[
\bigl(\overline{\sigma}_E(e\otimes\omega)\bigr)\,.\,a\ =\ 
\overline{\sigma}_E(e\otimes\omega)a
\oplus
(\mathrm{id}\otimes\sigma_E)\bigl(\overline{\sigma}_E(e\otimes\omega)\otimes\partial a\bigr).
\]
Therefore, it follows that
\[
\widetilde{\psi}\bigl((e\otimes\omega\oplus\xi)\,.\,a\bigr)\ =\
\widetilde{\psi}(e\otimes\omega\oplus\xi)\,.\,a.
\]
So $\widetilde{\psi}$ is also a right $\mathcal A$-module map.
\end{proof}

\begin{rmrk}
The above result is in fact true for any noncommutative complex curve. If one does not assume a bimodule splitting of the jet sequence, equivalently, if the existence of a bimodule $\partial$-connection $\nabla_E$ is not known beforehand, it is unclear to us at the moment whether the above extension result \emph{always} holds using only the data $(\overline{\sigma}_E,\,\overline{\nabla}_E,\,\sigma_E)$.
\end{rmrk}

In the situation of the holomorphic structure $\overline{\nabla}_J$ on $J_E^1$ constructed in \Pref{holomorphic structure}, a distinguished bimodule map $\psi_0$, as in \Pref{pp12}, turns out to be relevant in the case of $\mathbb{CP}_q^1$. For this specific map, it happens that one does not need to assume a bimodule splitting of the jet sequence beforehand; in particular, the existence of a $\partial$-connection on $E$ is not required.

Given the data $(\overline{\sigma}_E,\,\overline{\nabla}_E,\,\sigma_E)$ on $E$, consider the following linear map
\begin{equation}\label{reqd map}
\psi_0\ :=\ -(\mu^{-1}\otimes\mathrm{id}_E)\circ(\wedge\otimes
\mathrm{id}_E)\circ(\mathrm{id}_{\Omega^{1,0}}\otimes\overline{\sigma}_E)
\ :\ \bigl(\Omega^{1,0}\otimes_{\mathcal A}E\bigr)\otimes_{\mathcal A}\Omega^{0,1}
\ \longrightarrow\ \Omega^{0,1}\otimes_{\mathcal A}\bigl(\Omega^{1,0}\otimes_{\mathcal A}E\bigr)
\end{equation}
which is a $\mathcal{A}$-bimodule map. Recall from \Pref{holomorphic structure} the map
\[
\widetilde D_E\ :\ \Omega^{1,0}\otimes_{\mathcal A}E\ \longrightarrow\ 
\Omega^{0,1}\otimes_{\mathcal A}\bigl(\Omega^{1,0}\otimes_{\mathcal A}E\bigr),
\]
given by
\[
\widetilde D_E\ :=\
(\mu^{-1}\otimes \mathrm{id}_E)\circ(\overline{\partial}\otimes \mathrm{id}_E)
-
(\mu^{-1}\otimes \mathrm{id}_E)\circ(\wedge\otimes \mathrm{id}_E)\circ(\mathrm{id}_{\Omega^{1,0}}\otimes\overline\nabla_E).
\]
Also recall that, $\widetilde D_E$ is a left $\overline\partial$-connection.

\begin{lmma}\label{to be used here}
For all $a\,\in\,\mathcal A$, $e\,\in\, E$ and $\omega\,\in\,\Omega^{1,0}$,
\[
\widetilde D_E\bigl(\sigma_E(e\otimes\omega)\,.\,a\bigr)\ =\
\widetilde D_E\bigl(\sigma_E(e\otimes\omega)\bigr)a
+\psi_0\bigl(\sigma_E(e\otimes\omega)\otimes\overline{\partial} a\bigr),
\]
where $\psi_0$ is constructed in \eqref{reqd map}.
\end{lmma}

\begin{proof}
Write
\[
\sigma_E(e\otimes\omega)\ =\ \sum_j \xi_j\otimes e_j.
\]
Using the fact that $\overline{\nabla}_E$ is a right $\overline{\sigma}_E$-connection, we compute:
\begin{IEEEeqnarray*}{lCl}
& & D_E\bigl(\sigma_E(e\otimes\omega)\,.\,a\bigr)\\
&=& (\overline\partial\otimes\mathrm{id}_E)\bigl(\sigma_E(e\otimes\omega)\,.\,a\bigr)
-(\wedge\otimes\mathrm{id}_E)\circ(\mathrm{id}_{\Omega^{1,0}}\otimes\overline{\nabla}_E)\bigl(\sigma_E(e\otimes\omega)\,.\,a\bigr)\\
&=& (\overline\partial\otimes\mathrm{id}_E)\bigl(\sigma_E(e\otimes\omega)\bigr)a
-(\wedge\otimes\mathrm{id}_E)\Bigl(\sum_j \xi_j\otimes\overline{\nabla}_E(e_j)a+\xi_j\otimes\overline{\sigma}_E(e_j\otimes\overline\partial a)\Bigr)\\
&=& (\overline\partial\otimes\mathrm{id}_E)\bigl(\sigma_E(e\otimes\omega)\bigr)a
-\sum_j \xi_j\wedge\overline{\nabla}_E(e_j)a
-(\wedge\otimes\mathrm{id}_E)\circ(\mathrm{id}_{\Omega^{1,0}}\otimes\overline{\sigma}_E)\bigl(\sigma_E(e\otimes\omega)\otimes\overline\partial a\bigr)\\
&=& (\overline\partial\otimes\mathrm{id}_E)\bigl(\sigma_E(e\otimes\omega)\bigr)a
-(\wedge\otimes\mathrm{id}_E)\circ(\mathrm{id}_{\Omega^{1,0}}\otimes\overline{\nabla}_E)\bigl(\sigma_E(e\otimes\omega)\bigr)a\\
& & -(\wedge\otimes\mathrm{id}_E)\circ(\mathrm{id}_{\Omega^{1,0}}\otimes\overline{\sigma}_E)\bigl(\sigma_E(e\otimes\omega)\otimes\overline\partial a\bigr)\\
&=& D_E\bigl(\sigma_E(e\otimes\omega)\bigr)a
+(\mu\otimes\mathrm{id}_E)\circ\psi_0\bigl(\sigma_E(e\otimes\omega)\otimes\overline\partial a\bigr).
\end{IEEEeqnarray*}
Since
\[
\widetilde D_E\ =\ (\mu^{-1}\otimes\mathrm{id}_E)\circ D_E,
\]
the result follows.
\end{proof}

\begin{crlre}
If
\[
\sigma_E\ :\ E\otimes_{\mathcal A}\Omega^{1,0}\ \longrightarrow\ \Omega^{1,0}\otimes_{\mathcal A}E
\]
is surjective, then $\widetilde D_E$ is a bimodule $\psi_0$-connection.
\end{crlre}

Assume that $J_E^1$ is a bimodule over $\mathcal A$, so that we are given a bimodule map
\[
\sigma_E\ :\ E\otimes_{\mathcal A}\Omega^{1,0}\ \longrightarrow\ \Omega^{1,0}\otimes_{\mathcal A}E.
\]
Also, assume that $\overline{\nabla}_E$ is a $\overline{\sigma}_E$-bimodule $\overline\partial$-connection.

In the case of $\mathbb{CP}_q^1$, it is known that $\,\overline\partial\, :\,\mathcal A\,\longrightarrow\,\Omega^{0,1}$ is surjective \cite[Proposition~7.2]{DL}. We define a $\mathbb C$-linear map
\[
\sigma\ :\ E\otimes_{\mathcal A}\Omega^{0,1}\ \longrightarrow\ \Omega^{0,1}\otimes_{\mathcal A}\Omega^{1,0}\otimes_{\mathcal A}E
\]
that sends any $e\otimes\overline\partial a$ to
\[
(\mu^{-1}\circ\partial\otimes\mathrm{id}_E)\bigl(\overline{\sigma}_E(e\otimes\overline\partial a)\bigr)
+\widetilde D_E\bigl(\sigma_E(e\otimes\partial a)\bigr)
-(\mathrm{id}\otimes\sigma_E)\bigl(\overline{\nabla}_E(e)\otimes\partial a\bigr),
\]
which is well defined because $\ker(\overline\partial)=\ker(\partial)=\mathbb C$.

Now define
\begin{IEEEeqnarray}{lCl}\label{sigmaJ}
\sigma_J\ :\ J_E^1\otimes_{\mathcal A}\Omega^{0,1}\ \longrightarrow\ \Omega^{0,1}\otimes_{\mathcal A}J_E^1
\end{IEEEeqnarray}
by
\[
\sigma_J(e\otimes\overline\partial a\oplus\xi)\ =\
\overline{\sigma}_E(e\otimes\overline\partial a)\oplus\big(\psi_0(\xi)+ \sigma(e\otimes\overline\partial a)\big).
\]
Equivalently, with the identification
\[
J_E^1\otimes_{\mathcal A}\Omega^{0,1}\ \cong\
\bigl(E\otimes_{\mathcal A}\Omega^{0,1}\bigr)\oplus_{\mathbb C}
\bigl(\Omega^{1,0}\otimes_{\mathcal A}E\otimes_{\mathcal A}\Omega^{0,1}\bigr),
\]
we have the following diagram\,:
\[
\begin{tikzcd}[row sep=large, column sep=small]
J_E^1\otimes_{\mathcal{A}}\Omega^{0,1} \arrow[d, "\sigma_J"] & \equiv &
E\otimes_{\mathcal{A}}\Omega^{0,1} \arrow[d, "\overline{\sigma}_E"] \arrow[drr, "\sigma"] &
\oplus_{\mathbb C} &
\Omega^{1,0}\otimes_{\mathcal A}E\otimes_{\mathcal A}\Omega^{0,1} \arrow[d, "\psi_0"]\\
\Omega^{0,1}\otimes_{\mathcal A}J_E^1 & \equiv &
\Omega^{0,1}\otimes_{\mathcal A}E &
\oplus_{\mathbb C} &
\Omega^{0,1}\otimes_{\mathcal A}\Omega^{1,0}\otimes_{\mathcal A}E.
\end{tikzcd}
\]

\begin{lmma}\label{needed in next}
The map $\sigma_J$ in \eqref{sigmaJ} is left $\mathcal A$-linear if and only if
\begin{IEEEeqnarray}{lCl}\label{req1}
\sigma(be\otimes\overline\partial a)
 =\ b\,\sigma(e\otimes\overline\partial a)-\psi_0(\partial b\otimes e\otimes\overline\partial a),
\end{IEEEeqnarray}
and it is right $\mathcal A$-linear if and only if
\begin{IEEEeqnarray}{lCl}\label{req2}
\sigma(e\otimes(\overline\partial a)\,.\,b)\ =\ \sigma(e\otimes\overline\partial a)b
+(\mathrm{id}\otimes\sigma_E)\bigl(\overline{\sigma}_E(e\otimes\overline\partial a)\otimes\partial b\bigr).
\end{IEEEeqnarray}
\end{lmma}

\begin{proof}
For any  $b\,\in\, \mathcal A$, we have
\begin{IEEEeqnarray*}{lCl}
\sigma_J\bigl(b\,.\,(e\otimes\overline\partial a\oplus\xi)\bigr)
&=& \sigma_J\bigl(be\otimes\overline\partial a\oplus b\xi+\partial b\otimes e\otimes\overline\partial a\bigr)\\
&=& \overline{\sigma}_E(be\otimes\overline\partial a)\oplus\psi_0\bigl(b\xi+\partial b\otimes e\otimes\overline\partial a\bigr)+\sigma(be\otimes\overline\partial a)\\
&=& b\,\overline{\sigma}_E(e\otimes\overline\partial a)\oplus b\,\psi_0(\xi)
\oplus \psi_0(\partial b\otimes e\otimes\overline\partial a)\oplus \sigma(be\otimes\overline\partial a),
\end{IEEEeqnarray*}
while
\begin{IEEEeqnarray*}{lCl}
b\,.\,\sigma_J(e\otimes\overline\partial a\oplus\xi)
&=& b\,\overline{\sigma}_E(e\otimes\overline\partial a)\oplus b\,\psi_0(\xi)\oplus b\,\sigma(e\otimes\overline\partial a).
\end{IEEEeqnarray*}
This completes proof of \eqref{req1}.

Similarly,
\begin{IEEEeqnarray*}{lCl}
\sigma_J\bigl((e\otimes\overline\partial a\oplus\xi)\,.\,b\bigr)
&=& \sigma_J\bigl(e\otimes(\overline\partial a)\,.\,b\oplus \xi b\bigr)\\
&=& \overline{\sigma}_E(e\otimes\overline\partial a)b\oplus\psi_0(\xi)b\oplus \sigma(e\otimes(\overline\partial a)\,.\,b),
\end{IEEEeqnarray*}
whereas
\begin{IEEEeqnarray*}{lCl}
\sigma_J(e\otimes\overline\partial a\oplus\xi)\,.\,b
&=& \bigl(\overline{\sigma}_E(e\otimes\overline\partial a)\oplus\psi_0(\xi)\oplus\sigma(e\otimes\overline\partial a)\bigr)\,.\,b\\
&=& \overline{\sigma}_E(e\otimes\overline\partial a)b
\oplus \psi_0(\xi)b
\oplus \sigma(e\otimes\overline\partial a)b
\oplus (\mathrm{id}\otimes\sigma_E)\bigl(\overline{\sigma}_E(e\otimes\overline\partial a)\otimes\partial b\bigr).
\end{IEEEeqnarray*}
This completes proof of \eqref{req2}.
\end{proof}

We impose the following compatibility condition between the maps $\sigma_E$ and $\overline{\sigma}_E\,$:
\begin{IEEEeqnarray}{lCl}\label{22}
\psi_0\bigl(\sigma_E(e\otimes\partial a)\otimes\overline\partial b\bigr)\ =\
(\mathrm{id}\otimes\sigma_E)\bigl(\overline{\sigma}_E(e\otimes\overline\partial a)\otimes\partial b\bigr)
\end{IEEEeqnarray}
for all $e\,\in\, E$ and $a,\,b\,\in\,\mathcal A$, where $\psi_0$ is as defined in \eqref{reqd map}. Equivalently, the following diagram
\[
\begin{tikzcd}
 & E\otimes_{\mathcal{A}}\Omega^{1,0}\otimes_{\mathcal{A}}\Omega^{0,1} \arrow[r, "\sigma_E\otimes\mathrm{id}"] & \Omega^{1,0}\otimes_{\mathcal{A}}E\otimes_{\mathcal{A}}\Omega^{0,1} \arrow[dr, "\psi_0"] & \\
E\otimes\mathcal{A}\otimes\mathcal{A} \arrow[ur, "\mathrm{id}\otimes\partial\otimes\overline\partial"] \arrow[dr, "\mathrm{id}\otimes\overline\partial\otimes\partial" ']  & & & \Omega^{0,1}\otimes_{\mathcal{A}} \Omega^{1,0} \otimes_{\mathcal{A}}E \\
& E\otimes_{\mathcal{A}}\Omega^{0,1}\otimes_{\mathcal{A}}\Omega^{1,0} \arrow[r, "\overline{\sigma}_E\otimes\mathrm{id}"] & \Omega^{0,1}\otimes_{\mathcal{A}}E\otimes_{\mathcal{A}}\Omega^{1,0} \arrow[ur, "\mathrm{id}\otimes\sigma_E" '] &
\end{tikzcd}
\]
must commute. Notice that the compatibility condition is imposed only on the kernel part of the jet sequence, and not on $J_E^1\,$.

\begin{ppsn}
The map $\sigma_J$ in \eqref{sigmaJ} is a homomorphism of left $\mathcal{A}$-modules. Moreover, it is a right $\mathcal{A}$-module map if and only if \eqref{22} holds.
\end{ppsn}
\begin{proof}
We use \Lref{needed in next} to prove this. For any $b\,\in\,\mathcal{A}$, we have
\begin{IEEEeqnarray*}{lCl}
& & \sigma(be\otimes\overline\partial a)\\
&=& (\mu^{-1}\circ\partial\otimes\mathrm{id}_E)\circ\overline{\sigma}_E(be\otimes\overline\partial a)
+\widetilde D_E\bigl(\sigma_E(be\otimes\partial a)\bigr)
-(\mathrm{id}\otimes\sigma_E)\bigl(\overline{\nabla}_E(be)\otimes\partial a\bigr)\\
&=& (\mu^{-1}\circ\partial\otimes\mathrm{id}_E)\bigl(b\,\overline{\sigma}_E(e\otimes\overline\partial a)\bigr)
+\widetilde D_E\bigl(b\,\sigma_E(e\otimes\partial a)\bigr)
-b\,(\mathrm{id}\otimes\sigma_E)\bigl(\overline{\nabla}_E(e)\otimes\partial a\bigr)\\
& & -(\mathrm{id}\otimes\sigma_E)\bigl(\overline\partial b\otimes e\otimes\partial a\bigr)\\
&=& -b\,(\mathrm{id}\otimes\sigma_E)\bigl(\overline{\nabla}_E(e)\otimes\partial a\bigr)
-\overline\partial b\otimes \sigma_E(e\otimes\partial a)
+(\mu^{-1}\otimes\mathrm{id}_E)\circ D_E\bigl(b\,\sigma_E(e\otimes\partial a)\bigr)\\
& & +(\mu^{-1}\otimes\mathrm{id}_E)\Bigl(\partial b\wedge\overline{\sigma}_E(e\otimes\overline\partial a)
+b\,(\partial\otimes\mathrm{id})\bigl(\overline{\sigma}_E(e\otimes\overline\partial a)\bigr)\Bigr)\\
&=& -b\,(\mathrm{id}\otimes\sigma_E)\bigl(\overline{\nabla}_E(e)\otimes\partial a\bigr)
-\overline\partial b\otimes \sigma_E(e\otimes\partial a)
+(\mu^{-1}\otimes\mathrm{id}_E)\bigl(\partial b\wedge\overline{\sigma}_E(e\otimes\overline\partial a)\bigr)\\
& & +b\,(\mu^{-1}\otimes\mathrm{id}_E)\circ(\partial\otimes\mathrm{id})\bigl(\overline{\sigma}_E(e\otimes\overline\partial a)\bigr)
+b\,(\mu^{-1}\otimes\mathrm{id}_E)\circ D_E\bigl(\sigma_E(e\otimes\partial a)\bigr)\\
& & +(\mu^{-1}\otimes\mathrm{id}_E)\bigl(\overline\partial b\wedge\sigma_E(e\otimes\partial a)\bigr)\\
&=& -b\,(\mathrm{id}\otimes\sigma_E)\bigl(\overline{\nabla}_E(e)\otimes\partial a\bigr)
-\overline\partial b\otimes \sigma_E(e\otimes\partial a)
+(\mu^{-1}\otimes\mathrm{id}_E)\bigl(\partial b\wedge\overline{\sigma}_E(e\otimes\overline\partial a)\bigr)\\
& & +b\,(\mu^{-1}\otimes\mathrm{id}_E)\circ(\partial\otimes\mathrm{id})\bigl(\overline{\sigma}_E(e\otimes\overline\partial a)\bigr)\\
& & +b\,(\mu^{-1}\otimes\mathrm{id}_E)\Bigl((\overline\partial\otimes\mathrm{id}_E)\bigl(\sigma_E(e\otimes\partial a)\bigr)
-(\wedge\otimes\mathrm{id}_E)\circ(\mathrm{id}_{\Omega^{1,0}}\otimes\overline{\nabla}_E)\bigl(\sigma_E(e\otimes\partial a)\bigr)\Bigr)\\
& & +(\mu^{-1}\otimes\mathrm{id}_E)\bigl(\overline\partial b\wedge\sigma_E(e\otimes\partial a)\bigr)\\
&=& b\,\sigma(e\otimes\overline\partial a)
+(\mu^{-1}\otimes\mathrm{id}_E)\bigl(\overline\partial b\wedge\sigma_E(e\otimes\partial a)\bigr)
-\overline\partial b\otimes \sigma_E(e\otimes\partial a)\\
& & +(\mu^{-1}\otimes\mathrm{id}_E)\bigl(\partial b\wedge\overline{\sigma}_E(e\otimes\overline\partial a)\bigr)\\
&=& b\,\sigma(e\otimes\overline\partial a)-\psi_0(\partial b\otimes e\otimes\overline\partial a).
\end{IEEEeqnarray*}
Thus \eqref{req1} holds, and consequently, the map $\sigma_J$ in \eqref{sigmaJ} is a left $\mathcal{A}$-module map. Next, by \Cref{to be used here}, we have the following
\begin{IEEEeqnarray*}{lCl}
\sigma(e\otimes(\overline\partial a)\,.\,b)
&=&(\mu^{-1}\otimes\mathrm{id}_E)\circ(\partial\otimes\mathrm{id})\circ\overline{\sigma}_E\bigl(e\otimes(\overline\partial a)\,.\,b\bigr)\\
& & +\widetilde D_E\circ\sigma_E\bigl(e\otimes(\partial a)\,.\,b\bigr)
-(\mathrm{id}\otimes\sigma_E)\bigl(\overline{\nabla}_E(e)\otimes(\overline\partial a)\,.\,b\bigr)\\
&=& (\mu^{-1}\otimes\mathrm{id}_E)\circ(\partial\otimes\mathrm{id})\bigl(\overline{\sigma}_E(e\otimes\overline\partial a)\bigr)b
-(\mathrm{id}\otimes\sigma_E)\bigl(\overline{\nabla}_E(e)\otimes\overline\partial a\bigr)b\\
& & +\widetilde D_E\circ\sigma_E\bigl(e\otimes(\partial a)\,.\,b\bigr)\\
&=& \sigma(e\otimes\overline\partial a)b
+\widetilde D_E\bigl(\sigma_E(e\otimes(\partial a)\,.\,b)\bigr)
-\widetilde D_E\bigl(\sigma_E(e\otimes\partial a)\bigr)b\\
&=& \sigma(e\otimes\overline\partial a)b
+\psi_0\bigl(\sigma_E(e\otimes\partial a)\otimes\overline\partial b\bigr).
\end{IEEEeqnarray*}
Thus, \eqref{req2} holds, and consequently, the map $\sigma_J$ in \eqref{sigmaJ} is a right $\mathcal{A}$-module map if and only if the compatibility condition in \eqref{22} holds.
\end{proof}

\begin{thm}
Let $(E,\,\overline{\nabla}_E)$ be a holomorphic vector bundle over the quantum projective line. Assume that $E$ is a bimodule and the holomorphic structure $\overline{\nabla}_E$ on $E$ is a $\overline{\sigma}_E$-bimodule $\overline\partial$-connection. Suppose that the jet bundle $J_E^1$ is a bimodule, i.e.,\ there exists a bimodule map
\[
\sigma_E\ :\ E\otimes_{\mathcal{A}}\Omega^{1,0}\ \longrightarrow\
\Omega^{1,0}\otimes_{\mathcal{A}}E.
\]
If $\sigma_E$ and $\overline{\sigma}_E$ satisfy the compatibility condition in \eqref{22}, 
then the holomorphic structure $\overline{\nabla}_J$ on $J_E^1$ is also a bimodule connection.
\end{thm}

\begin{proof}
It was already shown that $\sigma_J$ is a bimodule map. We will show that the holomorphic structure $\overline{\nabla}_J$ from \Pref{holomorphic structure} is a right $\sigma_J$-connection on $J_E^1$, i.e.,
\begin{IEEEeqnarray}{lCl}\label{req3}
\overline{\nabla}_J\bigl((e\oplus\xi)\,.\,b\bigr)\ =\
\overline{\nabla}_J(e\oplus\xi)\,.\,b+\sigma_J\bigl((e\oplus\xi)\otimes\overline\partial b\bigr)
\end{IEEEeqnarray}
for all $b\,\in\,\mathcal A$, $e\,\in\, E$ and $\xi\,\in\,\Omega^{1,0}\otimes_{\mathcal A}E$. By definition of $\overline{\nabla}_J$,
\begin{IEEEeqnarray*}{lCl}
\overline{\nabla}_J\bigl((e\oplus\xi)\,.\,b\bigr)
&=& \overline{\nabla}_J\bigl(eb\oplus \xi b+\sigma_E(e\otimes\partial b)\bigr)\\
&=& \overline{\nabla}_E(eb)\oplus \widetilde{\beta}_E(eb)
\oplus \widetilde D_E(\xi b)
\oplus \widetilde D_E\bigl(\sigma_E(e\otimes\partial b)\bigr).
\end{IEEEeqnarray*}
On the other hand,
\begin{IEEEeqnarray*}{lCl}
& & \overline{\nabla}_J(e\oplus\xi)\,.\,b+\sigma_J\bigl((e\oplus\xi)\otimes\overline\partial b\bigr)\\
&=& \overline{\nabla}_J(e\oplus\xi)\,.\,b+\sigma_J(e\otimes\overline\partial b\oplus \xi\otimes\overline\partial b)\\
&=& \Bigl\{\bigl(\overline{\nabla}_E(e)\oplus \widetilde{\beta}_E(e)\oplus \widetilde D_E(\xi)\bigr)\,.\,b\Bigr\}
+\Bigl\{\overline{\sigma}_E(e\otimes\overline\partial b)\oplus \psi_0(\xi\otimes\overline\partial b)\oplus \sigma(e\otimes\overline\partial b)\Bigr\}\\
&=& \overline{\nabla}_E(e)b+\overline{\sigma}_E(e\otimes\overline\partial b)\\
& & \oplus\,(\mathrm{id}\otimes\sigma_E)\bigl(\overline{\nabla}_E(e)\otimes\partial b\bigr)
\oplus \widetilde{\beta}_E(e)b
\oplus \widetilde D_E(\xi)b
\oplus \psi_0(\xi\otimes\overline\partial b)
\oplus \sigma(e\otimes\overline\partial b).
\end{IEEEeqnarray*}
Since $\overline{\nabla}_E$ is a right $\overline{\sigma}_E$-connection,
\[
\overline{\nabla}_E(eb)\ =\ \overline{\nabla}_E(e)b+\overline{\sigma}_E(e\otimes\overline\partial b).
\]
Thus, the first component on both sides of \eqref{req3} agree.

For the second component we compute $\beta_E(eb)$ and $D_E(\xi b)$. First,
\begin{IEEEeqnarray*}{lCl}
\beta_E(eb)
&=& (\partial\otimes\mathrm{id}_E)\bigl(\overline{\nabla}_E(eb)\bigr)\\
&=& (\partial\otimes\mathrm{id}_E)\bigl(\overline{\nabla}_E(e)b+\overline{\sigma}_E(e\otimes\overline\partial b)\bigr)\\
&=& \beta_E(e)b+(\partial\otimes\mathrm{id}_E)\bigl(\overline{\sigma}_E(e\otimes\overline\partial b)\bigr).
\end{IEEEeqnarray*}
Next, writing $\xi\,=\,\sum_j \omega_j\otimes e_j$ with $\omega_j\,\in\,\Omega^{1,0}$, we obtain the following:
\begin{IEEEeqnarray*}{lCl}
D_E(\xi b)
&=& (\overline\partial\otimes\mathrm{id}_E)(\xi b)
-(\wedge\otimes\mathrm{id}_E)(\mathrm{id}_{\Omega^{1,0}}\otimes\overline{\nabla}_E)(\xi b)\\
&=& (\overline\partial\otimes\mathrm{id}_E)(\xi b)
-(\wedge\otimes\mathrm{id}_E)\Bigl(\sum_j \omega_j\otimes\overline{\nabla}_E(e_j)b+\omega_j\otimes\overline{\sigma}_E(e_j\otimes\overline\partial b)\Bigr)\\
&=& \sum_j \overline\partial(\omega_j)\otimes e_j b
-\Bigl(\sum_j \omega_j\wedge\overline{\nabla}_E(e_j)b+\omega_j\wedge\overline{\sigma}_E(e_j\otimes\overline\partial b)\Bigr)\\
&=& D_E(\xi)b-(\wedge\otimes\mathrm{id}_E)\circ(\mathrm{id}_{\Omega^{1,0}}\otimes\overline{\sigma}_E)(\xi\otimes\overline\partial b).
\end{IEEEeqnarray*}
Therefore, the second component of $\overline{\nabla}_J\bigl((e\oplus\xi)\,.\,b\bigr)$ is
\begin{IEEEeqnarray*}{lCl}
& & \widetilde{\beta}_E(eb)+\widetilde D_E(\xi b)+\widetilde D_E\bigl(\sigma_E(e\otimes\partial b)\bigr)\\
&=& \widetilde{\beta}_E(e)b
+(\mu^{-1}\circ\partial\otimes\mathrm{id}_E)\bigl(\overline{\sigma}_E(e\otimes\overline\partial b)\bigr)
+\widetilde D_E(\xi)b\\
& & -(\mu^{-1}\circ\wedge\otimes\mathrm{id}_E)\circ(\mathrm{id}_{\Omega^{1,0}}\otimes\overline{\sigma}_E)(\xi\otimes\overline\partial b)
+\widetilde D_E\bigl(\sigma_E(e\otimes\partial b)\bigr)\\
&=& \widetilde{\beta}_E(e)b+\widetilde D_E(\xi)b+\psi_0(\xi\otimes\overline\partial b)
+\sigma(e\otimes\overline\partial b)
+(\mathrm{id}\otimes\sigma_E)\bigl(\overline{\nabla}_E(e)\otimes\partial b\bigr),
\end{IEEEeqnarray*}
which is exactly the second component of
\[
\overline{\nabla}_J(e\oplus\xi)\,.\,b+\sigma_J\bigl((e\oplus\xi)\otimes\overline\partial b\bigr).
\]
Hence \eqref{req3} holds, and the proof is complete.
\end{proof}


\subsection{Holomorphic connection on line bundles over the quantum projective line}\label{Sec6.1}

We present an application of \Tref{thm:splitimpliesconnectionbothway}.
\smallskip

Let $\mathcal{O}(n)$ denote the line bundle of degree $n$ over the complex projective line $\mathbb{CP}^1$. A classical result of Atiyah says that holomorphic connection exists on $\mathcal{O}(n)$ if and only if $n=0$. Consider the holomorphic line bundles $\mathcal{L}_n$ over $\mathcal{A}(\mathbb{CP}_q^1)$, the coordinate algebra of the quantum projective line $\mathbb{CP}_q^1$, equipped with the \emph{standard} holomorphic structures $\overline{\nabla}^{(n)}_0$ obtained in \cite{KLvS11}. The word ``standard" is chosen purposefully here as it has been shown recently that there are more than one (in fact, countably many) holomorphic structures (up to gauge equivalence) on each line bundle $\mathcal{L}_n$ \cite{GLRM25}. This is a genuine noncommutative phenomenon with no classical analogue. Such holomorphic structures are referred to as \emph{non-standard holomorphic structures}, and these are intimately related with the existence of fixed points of certain nonlinear maps acting on an appropriate Banach space \cite{BGK26}.

We have already seen in \Cref{Sec4.1} that the $(1,0)$-connection $$\nabla^{\partial}(\cdot):=q^{-n-2}\,\omega_+(X_+\triangleright(\cdot))$$ on $\mathcal{L}_n$ constructed in \cite[Eqn. 3.30]{KLvS11} is a holomorphic connection only on $\mathcal{L}_0\,$. However, this does not rule out the possibility of the existence of a holomorphic connection on $\mathcal{L}_n$ for $n\neq 0$. To penetrate deeper, we need the notion of $SU_q(2)$-invariant connections. We briefly recall it here from \cite{Lan12}.

The coproduct $\Delta\,:\,\mathcal{A}(SU_q(2))\,\longrightarrow\,\mathcal{A}(SU_q(2))\otimes\mathcal{A}(SU_q(2))$ induces a left coaction on $\mathcal{A}(\mathbb{CP}_q^1):$
\[
\Delta_L\ :\ \mathcal{A}(\mathbb{CP}_q^1)\ \longrightarrow\ \mathcal{A}(SU_q(2))\otimes\mathcal{A}(\mathbb{CP}_q^1),
\]
as well as on line bundles:
\[
\Delta_{(n)}\ :\ \mathcal{L}_n\ \longrightarrow\ \mathcal{A}(SU_q(2))\otimes\mathcal{L}_n
\]
over $\mathbb{CP}_q^1$. The left coaction $\Delta_L$ extends to one forms
\[
\Delta_L^{(1)}\ :\ \Omega^1\ \longrightarrow\ \mathcal{A}(SU_q(2))\otimes\Omega^1,
\]
making the calculus left $SU_q(2)$-covariant. We then have the lift to cotangent-valued left coaction:
\[
\Delta_{(n)}^{(1)}\ :\ \Omega^1\otimes_{\mathcal{A}(\mathbb{CP}_q^1)}\mathcal{L}_n\ \longrightarrow\
\mathcal{A}(SU_q(2))\otimes(\Omega^1\otimes_{\mathcal{A}(\mathbb{CP}_q^1)}\mathcal{L}_n)
\]
given by
\[
\Delta_{(n)}^{(1)}(\omega\otimes\xi)\ :=\ \Delta_L^{(1)}(\omega)_{(-1)}\Delta_{(n)}(\xi)_{(-1)}\otimes\Delta_L^{(1)}(\omega)_{(0)}\otimes\Delta_{(n)}(\xi)_{(0)}.
\]

\begin{dfn}
A left connection $\nabla\,:\,\mathcal{L}_n\,\longrightarrow\,\Omega^1\otimes_{\mathcal{A}(\mathbb{CP}_q^1)}\mathcal{L}_n$ is called $SU_q(2)$-invariant if the following diagram
\[
\begin{tikzcd}
\mathcal{L}_n \arrow[d, "\Delta_{(n)}" '] \arrow[r, "\nabla"] & \Omega^1\otimes_{\mathcal{A}(\mathbb{CP}_q^1)}\mathcal{L}_n \arrow[d, "\Delta_{(n)}^{(1)}"]\\
\mathcal{A}(SU_q(2))\otimes\mathcal{L}_n \arrow[r, "\mathrm{id}\otimes\nabla"] & \mathcal{A}(SU_q(2))\otimes\big(\Omega^1\otimes_{\mathcal{A}(\mathbb{CP}_q^1)}\mathcal{L}_n\big)
\end{tikzcd}
\]
commutes.
\end{dfn}

On each line bundle $\mathcal{L}_n$, there is an $\mathcal{A}(\mathbb{CP}_q^1)$-valued Hermitian structure
\[
\langle\,.\,,\,.\rangle_n\ :\ \mathcal{L}_n\times\mathcal{L}_n\ \longrightarrow\ \mathcal{A}(\mathbb{CP}_q^1)
\]
using the isomorphism $\mathcal{L}_n\otimes_{\mathcal{A}(\mathbb{CP}_q^1)}\mathcal{L}_n^*\,\cong\,\mathcal{L}_0=\mathcal{A}(\mathbb{CP}_q^1)$. Since the Haar state on the $C^*$-algebra $C(SU_q(2))$ restricted to $\mathcal{A}(\mathbb{CP}_q^1)$ is $SU_q(2)$-invariant, when composing with $\langle\,.\,,\,.\rangle_n$, we obtain a $\mathbb{C}$-valued inner product on each line bundle $\mathcal{L}_n$ which is $SU_q(2)$-invariant. Now, equip $\mathcal{L}_n$ with the standard holomorphic structure $\overline{\nabla}_0^{(n)}$. Consider the $(1,\,0)$-connection $\nabla^\partial$ constructed in \cite{KLvS11}. It turns out that the total connection $\nabla_{\mathrm{tot}}^{(n)}\,:=\,\nabla^\partial+\overline{\nabla}_0^{(n)}$ on $\mathcal{L}_n$ is unitary, i.e.,
\[
\langle\nabla_{\mathrm{tot}}^{(n)}\,\phi,\,\psi\rangle_n+\langle\phi,\nabla_{\mathrm{tot}}^{(n)}\,\,\psi\rangle_n\ =\ d\langle\phi,\,\psi\rangle_n .
\]
Consider the space $\mathcal{C}(\mathcal{L}_n)$ of unitary connections on $\mathcal{L}_n$. We have an ``adjoint'' coaction
\[
\Delta^{\mathcal{C}}_{(n)}\ :\ \mathcal{C}(\mathcal{L}_n)\ \longrightarrow\ \mathcal{A}(SU_q(2))\otimes \mathcal{C}(\mathcal{L}_n)
\]
given by
\begin{IEEEeqnarray}{lCl}\label{coaction}
\Delta^{\mathcal{C}}_{(n)}(\cdot)\ :=\ \big(m\circ(S\otimes\mathrm{id})\big)_{12}\circ(\mathrm{id}\otimes\Delta_n^{(1)})\circ(\mathrm{id}\otimes(\cdot))\circ\Delta_{(n)}
\end{IEEEeqnarray}
with $S$ being the antipode of $SU_q(2)$, and $m$ is the multiplication map. The connections $\nabla$ on $\mathcal{L}_n$ that are fixed by this coaction are precisely the $SU_q(2)$-invariant unitary connections. We have the following result.

\begin{lmma}[\cite{Lan12}]\label{landi}
The canonical connection $\nabla_{\mathrm{tot}}^{(n)}$ on $\mathcal{L}_n$ is the unique invariant
connection for the coaction $\Delta^{\mathcal{C}}_{(n)}$ in \eqref{coaction}.
\end{lmma}

Therefore, if there exists a holomorphic connection $\nabla^{1,0}$ on $\mathcal{L}_n$, the total connection $\nabla^{1,0}+\overline{\nabla}_0^{(n)}$ is $SU_q(2)$-invariant if and only if $\nabla^{1,0}\,=\,\nabla^\partial$. This gives us the following result, completely analogous to the classical case of $\mathbb{C}\mathrm{P}^1$.

\begin{ppsn}\label{invariant}
Take $n\,\in\,\mathbb{Z}$ and consider the holomorphic line bundle $\mathcal{L}_n$ equipped with the standard holomorphic structure $\overline{\nabla}_0^{(n)}$. Under the requirement that the total connection $\nabla\,:\,\mathcal{L}_n\,\longrightarrow\,\Omega^1\otimes_{\mathcal{A}(\mathbb{CP}_q^1)}\mathcal{L}_n$ be $SU_q(2)$-invariant, the line bundle $\mathcal{L}_n$ admits a holomorphic connection if and only if $n\,=\,0$.
\end{ppsn}

\begin{proof}
This follows from the discussion at the beginning of the section, together with the fact (Lemma \ref{landi}) that the only $SU_q(2)$-invariant connection on $\mathcal{L}_n$ is $\nabla^{\partial}+\nabla^{\overline\partial}\,,$ if we already fix $\nabla^{\overline\partial}\,\equiv\,\overline{\nabla}^{(n)}_0$.
\end{proof}

We end this article with the following question:
\smallskip

\noindent\textbf{Question:} What happens when we equip $\mathcal{L}_n$ with non-standard holomorphic structures such as in \cite{GLRM25}? Is \Pref{invariant} still true in (all) those cases?

\bigskip

\bigskip

\noindent{\sc I. Biswas} (\texttt{indranil.biswas@snu.edu.in, inrdranil29@gmail.com})\\
{\footnotesize Department of Mathematics\\
Shiv Nadar University\\
Dadri 201314, Uttar Pradesh, India}
\vspace*{1cm}

\noindent{\sc S. Guin} (\texttt{sguin@iitk.ac.in})\\
{\footnotesize Department of Mathematics\\
Indian Institute of Technology, Kanpur\\
Uttar Pradesh 208016, India}

\vspace*{1cm}

\noindent{\sc P. Kumar} (\texttt{pradip.kumar@snu.edu.in})\\
{\footnotesize Department of Mathematics\\
Shiv Nadar University\\
Dadri 201314, Uttar Pradesh, India}

\end{document}